\documentclass[11pt, reqno]{amsart}

\usepackage{textcomp} 
\usepackage{amsmath, amsfonts,amsthm,amssymb,amsxtra}
\usepackage{amscd}
\usepackage{enumitem}
\usepackage{url}
\usepackage{tikz-cd} 
\usepackage{xcolor}
\usepackage{graphicx}
\usepackage{verbatim}

\newtheorem{theorem}{Theorem}[section]
\newtheorem{proposition}[theorem]{Proposition}

\newtheorem{corollary}[theorem]{Corollary}
\newtheorem{question}[theorem]{Question}

\theoremstyle{definition}

\newtheorem{examples}[theorem]{Example}

\theoremstyle{remark}

\newtheorem{remark}[theorem]{Remark}

\numberwithin{equation}{section}

\renewcommand{\epsilon}{\varepsilon}

\renewcommand{\phi}{\varphi}

\begin{document}

{
\title{Maximal contact and symplectic structures}
\author{Oleg Lazarev}
\address{Oleg Lazarev}
\address{Columbia University}
\email{olazarev@math.columbia.edu}  
}

\maketitle

\begin{abstract}

We introduce a procedure for 
gluing Weinstein domains along Weinstein subdomains. By gluing along flexible subdomains, we show that any finite collection of high-dimensional Weinstein domains with the same topology are Weinstein subdomains of a `maximal'
Weinstein domain also with the same topology.  As an application, we produce exotic cotangent bundles containing many closed regular Lagrangians
that are formally Lagrangian isotopic but not Hamiltonian isotopic and also  give a new construction of exotic Weinstein structures on Euclidean space. 
We describe a similar construction in the contact setting which we use to produce `maximal' contact structures and extend several existing results in low-dimensional contact geometry to high-dimensions. We prove that all contact manifolds have symplectic caps, introduce a general procedure for producing contact manifolds with many Weinstein fillings, and give a new proof of the existence of codimension two contact embeddings. 
\end{abstract}

\section{Introduction}\label{sec: intro}
\subsection{The category of Weinstein domains}\label{ssec: cat_weinstein}

One of the main problems in symplectic topology is to classify all symplectic structures on a given smooth manifold. 
In this paper, we will focus on \textit{Weinstein domains}, which are exact symplectic manifolds equipped with a Morse function compatible with the symplectic structure. These domains encompass a large class of exact symplectic manifolds, like cotangent bundles and affine varieties. 
There has been significant progress on classifying Weinstein domains in dimension 4. For example, there is a unique Weinstein structure on $B^4$ and  $T^*S^2$ while there is no Weinstein structure on $S^2 \times D^2$ \cite{CE12}, although it has the necessary smooth topology.
 On the other hand, there is no such classification result for any high-dimensional domain. Furthermore, any high-dimensional smooth domain with the appropriate topology has infinitely many different Weinstein structures, distinguished by the Floer-theoretic invariant  \textit{symplectic cohomology} \cite{MM, CE12}. To further complicate matters, there are Weinstein domains with vanishing symplectic cohomology over certain finite fields but not over the integers \cite{abouzaid_seidel_recombination} or with vanishing symplectic cohomology over the integers but non-vanishing twisted symplectic cohomology \cite{MS}.

Fortunately, there is a natural relationship on the set of Weinstein domains which sheds some light on the classification problem.
This relationship is given by \textit{Weinstein cobordisms}, which are exact symplectic cobordisms that have Morse functions compatible with the symplectic structure. A Weinstein cobordism $W$ has a  splitting of its boundary $\partial W$ into positive and negative components $\partial_+ W, \partial_- W$  on which the Morse function is increasing, decreasing respectively in
an outward direction; see Section \ref{sec: background}. 
These boundaries have natural contact structures and if $W_1, W_2$ are two Weinstein cobordisms such that $\partial_+ W_1, \partial_- W_2$ agree, then we can glue $W_1, W_2$ along this contact manifold to get a Weinstein cobordism $W_1 \circ W_2$. We call this gluing operation \textit{concatenation}. Note that concatenation does not actually require $W_i$ to be Weinstein and also works if $W_i$ are \textit{Liouville} cobordisms, i.e. exact symplectic cobordisms with contact-type boundary, possibly without a compatible Morse function. 
We can also concatenate Weinstein cobordisms to Weinstein domains.
If $W_1$ is a Weinstein domain and $C$ is a Weinstein cobordism with matching contact boundaries, then $W_1 \circ C$ is another Weinstein domain $W_2$; equivalently, $W_1$ is a Weinstein subdomain of $W_2$.
If there is a Weinstein cobordism from domain $W_1$ to $W_2$ and from $W_2$ to $W_3$, then there is also a Weinstein cobordism from $W_1$ to $W_3$ obtained by concatenating.  

By concatenating Weinstein cobordisms, we can enrich the set of Weinstein domains to a category. More precisely, let $\mathfrak{Weinstein}$ be the category whose objects are Weinstein domains and whose morphisms are given by Weinstein cobordisms between Weinstein domains, with composition of morphisms given by concatenation; both domains and cobordism will be up to Weinstein homotopy. Since we are mainly interested in symplectic structures that are the same from the point of view of smooth topology, we will often fix an \textit{almost} symplectic structure  $(W^{2n}, J)$, an almost complex structure $J$ on $W^{2n}$. Let $\mathfrak{Weinstein}(W, J)$ be the subcategory of $\mathfrak{Weinstein}$ whose objects are Weinstein domains in the almost symplectomorphism class $(W,J)$ and whose morphisms are smoothly trivial Weinstein cobordisms; gluing a smoothly trivial cobordism to a Weinstein domain does not change the almost symplectomorphism type. As we will see, many existing results about Weinstein domains can be reformulated in terms of the Weinstein categories $\mathfrak{Weinstein}$ and $\mathfrak{Weinstein}(W,J)$. In this paper, we investigate some properties of these categories and give applications. We also note that there is a different category in symplectic topology that is also called the Weinstein category \cite{Weinstein_category}, whose objects are (closed) symplectic manifolds and whose morphisms are Lagrangian correspondences; however this more general category seems less suited for the geometric applications we have in mind, like Corollaries \ref{cor: finitely_lagrangians}, 
\ref{cor: exotic_lagrangians}.

One important property of Weinstein cobordisms is that they induce maps on J-holomorphic curve invariants like symplectic cohomology. More precisely, Viterbo \cite{viterbo1} showed that if $W_0$ is a Weinstein subdomain of $W_1$, then there is a transfer map $SH(W_1) \rightarrow SH(W_0)$ on symplectic cohomology. 
That is, symplectic cohomology $SH$ is a contravariant functor from $\mathfrak{Weinstein}$ to the category of unital rings. More generally, symplectic cochains gives a functor from $\mathfrak{Weinstein}$ to $L_\infty$-algebras and the wrapped Fukaya category gives a functor to triangulated categories \cite{Abouzaid_Seidel} but in this paper we will focus mainly on the symplectic cohomology functor since the implications of such a functor are already quite interesting. Via this functor, certain properties of the ring category can be transferred to the Weinstein category. Since J-holomorphic curves and associated functors like symplectic cohomology are the only known source of rigidity of the Weinstein category, one guiding question is to what extent the Weinstein category differs from the category of rings. Although there are different Weinstein domains with the same symplectic cohomology \cite{MS}, we will show that certain properties of the category of rings have shadows in the Weinstein setting. In particular, we will focus more on the morphisms of the Weinstein category than the objects.

The category of rings has a terminal object, the zero ring, and any ring has a unique map to that ring. This trivial algebraic statement has an interesting geometric analog. Cieliebak and Eliashberg \cite{CE12} showed that if $(W^{2n},J)$ is \textit{almost Weinstein}, i.e. $W^{2n}$ has a smooth Morse function whose critical points have index at most $n$, and $n \ge 3$, then there is a special  \textit{flexible} Weinstein structure $W_{flex}$ 
in $\mathfrak{Weinstein}(W,J)$. This flexible structure has vanishing symplectic cohomology \cite{MS} and satisfies an h-principle \cite{CE12}, meaning that its symplectic topology reduces to the underlying algebraic topology.  
In previous work, the author \cite{Lazarev_critical_points}
proved that $W_{flex}$ is a minimal element in $\mathfrak{Weinstein}(W, J)$. Namely, for any Weinstein structure $W^{2n}, n \ge 3,$ there is a smoothly trivial Weinstein cobordism $C^{2n}$ from $W_{flex}^{2n}$ to $W^{2n}$; see Theorem \ref{thm: prev_crit_points} below. 
However unlike in the category of rings which has a unique terminal object,  it is not known whether the flexible structure is the only minimal object in $\mathfrak{Weinstein}(W,J)$ or whether there can exist  non-flexible minimal elements, e.g. subflexible exotic Weinstein balls. On a related note, the category of unital rings is not symmetric: any non-zero ring has a unital ring map to the zero ring but not conversely. Similarly, the Weinstein category is not symmetric: the existence of a Weinstein cobordism from $W_0$ to $W_1$ does not imply the existence of a Weinstein cobordism from $W_1$ to $W_0$. 
So directionality is quite important in the Weinstein setting, unlike in the smooth case.

Another basic property of the category of rings is the existence of pullbacks. Hence it is natural to ask if the Weinstein category has pushouts, which are dual to pullbacks. Although we do not know if pushouts exist in this category, we will show that there are objects that satisfy a certain extension property similar to that of a pushout. Geometrically, this extension corresponds to  a certain gluing  
of Weinstein domains along subdomains, or Weinstein cobordisms along contact manifolds. 
As we explained, two Liouville cobordisms $W_1, W_2$  with $\partial_+ W_1= \partial_- W_2$ can be concatenated to produce $W_1 \circ W_2$, which corresponds to composition of morphisms. 
Weinstein cobordisms have an additional property which will be crucial in this paper:
they can be glued if their \textit{negative} boundaries $\partial_-W_1, \partial_-W_2$ agree, 
without the condition $\partial_+W_1 = \partial_-W_2$ on the positive boundary of $W_1$.
We will call this operation \textit{stacking} to distinguish it from concatenation.  
\begin{theorem}\label{thm: stacking}
Suppose $W_1, \cdots, W_k$ are Weinstein cobordisms  such that $\partial_- W_i = (Y, \xi)$ for all $i$. Then there are Weinstein cobordisms $C_i$ with $\partial_- C_i = \partial_+ W_i$ such that $W_i \circ C_i$ are all Weinstein homotopic. 
 If $W_i^{2n}$ are smoothly trivial and $n \ge 3$, then so are $C_i^{2n}$. 
\end{theorem}
See Theorem \ref{thm: stacking2} for a more precise version. 
We can take $(Y, \xi)$ to be the positive end of some Weinstein domain $W_0$ and form the Weinstein homotopic domains $W_0 \circ W_i \circ C_i$. 
Then Theorem \ref{thm: stacking} can be formulated in categorical terms as follows: every diagram in $\mathfrak{Weinstein}(W^{2n},J), n \ge 3,$ of the form
\begin{equation}\label{diag: comm1}
\begin{tikzcd}
W_1  &  \\
W_0 \arrow{r}{\phi_{02}} \arrow[swap]{u}{\phi_{01}} & W_2
\end{tikzcd}
\end{equation}
can be extended to a commutative diagram
\begin{equation}\label{diag: comm2}
\begin{tikzcd} 
W_1  \arrow{r}{\phi_{13}} & W_3  \\
W_0 \arrow{r}{\phi_{02}} \arrow[swap]{u}{\phi_{01}} & W_2\arrow{u}{\phi_{23}}
\end{tikzcd}
\end{equation}
for some Weinstein domain $W_3 \in \mathfrak{Weinstein}(W,J)$. 
A similar result holds in the category $\mathfrak{Weinstein}$, without any restrictions on $n$. The domain $W_3$ is essentially $W_1, W_2$ glued along the subdomain $W_0$. However $W_3$ is not unique. There are many domains that can be fit into the commutative diagram Figure \ref{diag: comm2} and in fact, the stacking construction itself involves many choices; see Section \ref{ssec: concatenate_stacking}. We do not know whether there exists a ``best" possible $W_3$, i.e. whether the Weinstein category actually has pushouts, corresponding to the fact that the category of rings has pullbacks. Ganatra, Pardon, and Shende \cite{Ganatra_Pardon_Shende_ii} recently considered a related but different gluing operation along Weinstein \textit{sectors}, which are Weinstein domains equipped with the extra data of a Weinstein hypersurface in their contact boundary. 
Inclusions of Weinstein sectors induce \textit{covariant} maps on J-holomorphic curve invariants while the inclusions of  Weinstein subdomains considered in this paper induce contravariant maps.

One important aspect of the stacking construction in Theorem \ref{thm: stacking} is that it 
uses Weinstein-ness in a crucial way, unlike the concatenation construction for Liouville cobordisms. 
Hence we do not know whether the extension property above holds for the category of almost symplectomorphic Liouville domains, with morphisms given by smoothly trivial Liouville cobordisms; see Question \ref{question: liouville_stacking}. Understanding the difference between Liouville  and Weinstein domains is a major open problem in symplectic topology. 

Now we discuss some applications of the stacking construction to maximal Weinstein structures. As we explained above, the flexible structure $W_{flex}$ is a minimal element in $\mathfrak{Weinstein}(W,J)$. Given the flexibility of minimal elements,  one would expect a maximal element of  $\mathfrak{Weinstein}(W,J)$ to display rigidity and have rich J-holomorphic curve invariants; it would also seem that such an element could not be produced via an h-principle. Although we do not know whether there is a maximal Weinstein domain that contains \textit{all} other almost symplectomorphic domains as subdomains, we can use the stacking construction to show that any \textit{finite} collection of Weinstein domains does have a maximal element. Even though such maximal elements do have interesting J-holomorphic curve invariants, their construction relies crucially on flexibility methods, in particular the existence of minimal elements and the h-principle for flexible Weinstein structures. 
\begin{theorem}\label{thm: finitely_many_subdomains}
	For any almost symplectomorphic Weinstein domains $W_1^{2n}, \cdots, W_k^{2n}, n \ge 3,$ there exists an almost symplectomorphic Weinstein domain $W^{2n}$ 
	such that all $W_i^{2n}$ are Weinstein subdomains 
	of $W^{2n}$
	and $W^{2n}\backslash W_i^{2n}$ is a smoothly trivial Weinstein cobordism. 
\end{theorem}
See Theorem \ref{thm: finitely_many_subdomains2}. The key point of Theorem \ref{thm: finitely_many_subdomains}
is that $W^{2n}$ is almost symplectomorphic to the original domains $W_i^{2n}$. Indeed, if we drop this condition, then we can just take $W$ to be the boundary connected sum $W_1^{2n} \natural \cdots \natural W_k^{2n}$, which clearly contains all $W_i^{2n}$ as subdomains. This construction works if $W_i$ are Weinstein structures on the standard ball but not for domains with more general topology. 
Hence the construction in Theorem \ref{thm: finitely_many_subdomains} is a kind of generalized boundary connected sum operation that does not change the smooth topology of the Weinstein domain
and is closely related to the ``pushout" in Theorem \ref{thm: stacking}. 
However since Theorem \ref{thm: finitely_many_subdomains} uses an h-principle, we do not have an explicit description of this maximal construction. 

Although we do know whether there is a maximal Weinstein domain that contains \textit{all} other almost symplectomorphic Weinstein domains, there is a maximal Weinstein \textit{manifold},  whose compatible Morse function  has possibly infinitely many critical points. 
 \begin{corollary}\label{cor: infinite_type}
	For any almost Weinstein domain $(W^{2n},J), n\ge 3$, there exists a maximal Weinstein manifold $W_{max}^{2n}$ almost symplectomorphic to the completion of $(W, J)$ such that any Weinstein domain almost symplectomorphic to $(W, J)$ 
	is a  subdomain of $W^{2n}_{max}$. 
\end{corollary}
See Corollary \ref{cor: infinite_type2}. 
We also investigate the converse to
Theorems \ref{thm: stacking},  \ref{thm: finitely_many_subdomains} and show that all domains can be obtained by stacking \textit{flexible} Weinstein domains.
In fact, a Weinstein domain obtained by stacking subdomains can be covered by those subdomains, implying the following result. 
\begin{corollary}\label{cor: covering}
For any Weinstein domain $W^{2n}, n \ge 3,$ with $\pi_1(W) = 0$, there exist two Weinstein subdomains $\phi_1, \phi_2: W_{flex}^{2n} \hookrightarrow W^{2n}$ such that 
$\phi_1(W_{flex}^{2n})\cup \phi_2(W_{flex}^{2n}) = W^{2n}$.
\end{corollary}
See Corollary \ref{cor: domain_flex_everything}. 
Here the Weinstein subdomains $\phi_i(W_{flex})$ can intersect the boundary of $W$. Also, the sources of $\phi_1, \phi_2$ are only \textit{Weinstein homotopic} to $W_{flex}$ and may not actually be the same domains. The number of flexible embeddings in Corollary \ref{cor: covering} needed to cover $W$ is sharp
since if $W$ is covered by a single subdomain, then it agrees with that subdomain. 
\begin{examples}
Any Weinstein domain $\Sigma^{2n}, n \ge 3,$ almost symplectomorphic to 
$B^{2n}_{std}$ can be covered by two Weinstein subdomains $B^{2n}_1, B^{2n}_2$ which are Weinstein homotopic to $B^{2n}_{std}$. 
\end{examples}
So this type of symplectic  Lusternik-Schnirelmann category does not have much symplectic information, at least when we cover by subdomains that are Weinstein homotopic to Darboux balls. This category may be more interesting if the covering is by actual Darboux balls.

The maximal construction in Theorem \ref{thm: finitely_many_subdomains} can be viewed as an operation on the set $\mathfrak{Weinstein}(W, J)$ that combines existing Weinstein domains to produce a new Weinstein domains with certain desired properties. In particular, it can be used to produce \textit{exotic} Weinstein domains, which are the same from the point of view of smooth topology but not symplectic topology. Although exotic Weinstein domains are an example of symplectic rigidity,  our construction of such domains via Theorem \ref{thm: finitely_many_subdomains} will rely on h-principles from the flexible side of symplectic topology. 
For example, Theorem \ref{thm: finitely_many_subdomains} can be used to produce Weinstein domains with many closed exact Lagrangians. 

We first explain how to construct Weinstein domains with many Lagrangians that have different smooth topology. For simplicity, we  assume that the ambient Weinstein domain is almost symplectomorphic to $T^*S^n$ although there are similar results for more general domains.  Any closed formal Lagrangian $L$ in $T^*S^n$, with possibly some exotic Weinstein structure, must have trivial complexified cotangent bundle $T^*L^n\otimes \mathbb{C}$. If the Lagrangian is \textit{regular}, i.e. has Weinstein complement,
then $[L] \in H_n(T^*S^n) \cong \mathbb{Z}$ is $\pm 1$; for $n$ even this implies that $\chi(L) = 2$.   In previous work, Eliashberg, Ganatra, and the author \cite{EGL} proved the converse: for any closed smooth manifold $L^n, n \ge 3$ and even, such that $L^n$ is orientable, $T^*L \otimes \mathbb{C}$ is trivial, and $\chi(L) = 2$, there exists an exotic Weinstein structure $T^*S^n_L$ almost symplectomorphic to $T^*S^n$ that contains $L$ as a regular Lagrangian. We extend their result to the case of multiple Lagrangians. 
\begin{corollary}\label{cor: finitely_lagrangians}
	For any closed orientable manifolds $L_1^n, \cdots, L_k^{n}, n \ge 3$ and even, with trivial $T^*L_i^n\otimes \mathbb{C}$ and $\chi(L_i^n) =2$,  there exists a Weinstein domain $T^*S^n_{L_1, \cdots, L_k}$ almost symplectomorphic to $T^*S^n_{std}$ that contains $L_1^n, \cdots, L_k^n$ as regular Lagrangians. 
\end{corollary}
\begin{proof}
	By 
	Theorem \ref{thm: finitely_many_subdomains} there is a Weinstein domain $T^*S^n_{L_1, \cdots, L_k}$ almost symplectomorphic to $T^*S^n_{std}$ 
	that contains all of $T^*S^n_{L_1}, \cdots, T^*S^n_{L_k}$ as Weinstein subdomains.
\end{proof}
In particular, there exist exotic cotangent bundles that contain arbitrarily many non-homotopy-equivalent closed exact Lagrangians. 
However it is not known whether there exist Weinstein domains  with  \textit{infinitely many}  such Lagrangians. 
We note that a version of Corollary \ref{cor: finitely_lagrangians} holds for $n$ odd with slightly different topological conditions on $L^n_i$. By gluing another Weinstein cobordism to $T^*S^n_{L_1, \cdots, L_k}$,  
it is possible extend Corollary \ref{cor: finitely_lagrangians} to Weinstein domains with more general smooth topology. 

A variation on the maximal construction can be used to produce Weinstein domains with many closed, exact Lagrangians that have the \textit{same} topology. There has been much recent work  on this problem. For example, \cite{Vianna_infinite_tori_cp2, Vianna_del_pezzo, Shende_treumann_williams_combinatorics_Lag_surf} produce infinitely many Lagrangian tori in certain closed symplectic 4-manifolds like $\mathbb{CP}^2$ using a construction motivated by mirror symmetry and related to cluster algebras; in fact, these are infinitely many regular Lagrangian tori in $\mathbb{CP}^2\backslash E$, where $E$ is the elliptic curve.
Our method produces only finitely many Lagrangians but with essentially arbitrary topology. Furthermore, we can control the smooth topology of the ambient Weinstein domain. However we cannot hope to control the symplectic topology of this domain  since for certain special Weinstein domains, e.g. cotangent bundles, there are strong restrictions on the topology of their Lagrangians. 
\begin{corollary}\label{cor: exotic_lagrangians}
For any closed manifold $M^n, n \ge 3,$ and any $k$, there is a Weinstein domain  $W^{2n}_k$ almost symplectomorphic to $T^*M^n_{std}$ and $k$ regular Lagrangians embeddings of $M^n$ into $W^{2n}$ that are formally Lagrangian isotopic but are not pair-wise Hamiltonian isotopic. 
\end{corollary}
See Corollary \ref{cor: exotic_lagrangians2}. As we will explain in the proof, these $k$ closed Lagrangians are distinguished by a `test' Lagrangian disk in $W^{2n}$ that has different wrapped Floer cohomology with each of them.  From the categorical point of view, we expect that the regular Lagrangian embeddings of $M^n$ into $W^{2n}$ give rise to many different morphisms from $T^*M^n_{std}$ to $W^{2n}$ in  $\mathfrak{Weinstein}(T^*S^n)$; however actually proving this seems to require the stronger statement that the $k$ Lagrangians are not related by a \textit{symplectomorphism} of $W^{2n}$. On the other hand, by results \cite{FukSS} on the nearby Lagrangian conjecture, there cannot be any morphisms from $W^{2n}$ to $T^*M_{std}$, another example of 
asymmetry in $\mathfrak{Weinstein}$.  
Corollary \ref{cor: exotic_lagrangians} can be combined with Corollary \ref{cor: finitely_lagrangians} to produce exotic cotangent bundles with a whole zoo of closed regular Lagrangians, including those that have different topology or have the same topology but are not Hamiltonian isotopic. So the construction of exotic Weinstein domains with many different Lagrangians is also quite flexible.

\subsection{The category of contact structures}\label{ssec: cat_contact}

The boundaries of a Weinstein cobordism have natural contact structures and hence there are contact analogs of the categories 
$\mathfrak{Weinstein}$ and  $\mathfrak{Weinstein}(W, J)$. Let  $\mathfrak{Contact}$ be the category whose objects are contact structures and morphisms are Weinstein cobordisms; see \cite{JP} for a similar category whose morphisms are \textit{Liouville} cobordisms. For an \textit{almost} contact structure $(Y, J)$, let $\mathfrak{Contact}(Y, J)$ be the subcategory of $\mathfrak{Contact}$ whose objects are contact structures almost contactomorphic to $(Y, J)$ and morphisms are smoothly trivial Weinstein cobordisms.
As we will explain, many classical results and problems in contact topology concern the categories $\mathfrak{Contact}$ and $\mathfrak{Contact}(Y,J)$. 

We first note that the contact and Weinstein categories have some similarities. Like $\mathfrak{Weinstein}$, the category $\mathfrak{Contact}$  satisfies a similar extension property as in Diagrams \ref{diag: comm1}, \ref{diag: comm2} given by the stacking construction; in fact, this is precisely the statement of Theorem \ref{thm: stacking}. Furthermore the category $\mathfrak{Contact}$ has also a functor to the category of unital rings given by a certain J-holomorphic curve invariant called \textit{contact homology} \cite{EGH, JP}. 
There is a functor from $\mathfrak{Weinstein}(W, J)$ to $\mathfrak{Contact}(\partial W, J)$ obtained by restricting a Weinstein domain to its contact boundary; this functor 
is faithful but is not full nor essentially surjective. Furthermore, the symplectic cohomology and contact homology functors do not commute with this restriction functor.

Analogous to the flexible structure $W_{flex}$ in  $\mathfrak{Weinstein}(W,J)$, there exists a special \textit{overtwisted} structure in $\mathfrak{Contact}(Y, J)$  constructed by Borman, Eliashberg, and Murphy \cite{BEMtwisted}. This structure also satisfies an h-principle and has vanishing contact homology \cite{Bourgeois_van_Koert_overtwisted}. 
Casals, Murphy, and Presas \cite{CMP}  showed that 
for $n \ge 3$, the overtwisted structure $(Y^{2n-1}, \xi_{ot})$ is a minimal element in $\mathfrak{Contact}(Y^{2n-1}, J)$; for $n = 2$, $\mathfrak{Contact}(Y^3,J)$ might not have minimal elements but Etnyre and Honda \cite{Etnyre_Honda_caps} showed that  \textit{any} overtwisted contact 3-manifold $(Y^3, \xi_{ot})$ is minimal in $\mathfrak{Contact}$. Our main result for contact manifolds is the existence of maximal contact structures, analogous to the existence of maximal Weinstein structures in Theorem \ref{thm: finitely_many_subdomains}. Again, the proof relies on flexibility results like the h-principle for overtwisted contact structures \cite{BEMtwisted}. 
\begin{theorem}\label{thm: contact_cobordisms}
	For any contact manifolds $(Y_1^{2n-1}, \xi_1), \cdots, (Y_k^{2n-1}, \xi_k), n \ge 3,$ there exist Weinstein cobordisms $C_i^{2n}$ such that 
	$\partial_-C_i^{2n} = (Y_i^{2n-1}, \xi_i)$ and $\partial_+C_i^{2n}$ are all contactomorphic. If $(Y_i^{2n-1}, \xi_i)$ are almost contactomorphic, then $C_i^{2n}$ 
	can be taken to be smoothly trivial.  
\end{theorem}
See Theorem \ref{thm: contact_cobordisms2} 
and Corollary \ref{cor: contact_cobordisms_diff_top} for the proof. 
Theorem \ref{thm: contact_cobordisms} is a geometric version of a result of 
Bowden, Crowley, and Stipsicz \cite{Bowden_Crowley_Stipsicz_stein_fillable_II} for maximal  \textit{almost} contact structures: there is an almost contact manifold $(M_{max}, J_{max})$ such that for any almost contact manifold $(Y, J)$ (with possibly different topology), there is an almost Weinstein cobordism $(W^{2n}, J)$ with $\partial_- (W, J) = (Y, J)$ and $\partial_+ (W,J)= (M_{max}, J_{max})$.  The first claim in Theorem \ref{thm: contact_cobordisms} for arbitrary contact manifolds can be reduced to the second claim in Theorem \ref{thm: contact_cobordisms} 
for almost contactomorphic structures by this result of Bowden, Crowley, and Stipsicz \cite{Bowden_Crowley_Stipsicz_stein_fillable_II}. 
We also prove a similar result in dimension 3, although the Weinstein cobordisms $C_i$ are no longer smoothly trivial; see Theorem \ref{thm: contact_cob_dimension4}. 
As in the Weinstein setting, we do not know an explicit description of the maximal construction in terms of contact surgery presentations. 
For example, this maximal construction is quite different from the usual contact connected sum operation, even when $(Y^{2n-1}, J)$ is $(S^{2n-1}, J_{std})$: if $(Y_1, \xi_1)$ is overtwisted, then $(Y_1, \xi_1) \sharp (Y_2, \xi_2)$ is also overtwisted and hence if $(Y_2, \xi_2)$ is fillable, there is no Weinstein cobordism from $(Y_2, \xi_2)$ to $(Y_1, \xi_1) \sharp (Y_2, \xi_2)$ (although there is a natural Weinstein cobordism from $(Y_1, \xi_1) \coprod (Y_2, \xi_2)$ to $(Y_1, \xi_1) \sharp (Y_2, \xi_2)$). 

The maximal construction  Theorem \ref{thm: contact_cobordisms} provides a uniform approach to several problems in contact topology. Our first application is a structure result for contact manifolds. Given a framed isotropic sphere in a contact manifold, there is a procedure called \textit{contact surgery} that produces a new contact manifold and  for a Legendrian sphere $\Lambda$, there is a similar \textit{anti-surgery} procedure; see Section \ref{sec: background} for details. Contact surgery and anti-surgery presentations, along with open book decompositions, are the main explicit models for contact manifolds and are quite useful for calculations. Ding and Geiges \cite{Ding_Geiges_contact_surgerypresentation_dim3} showed that all contact 3-manifolds can be obtained via contact surgery and anti-surgery on $(S^3, \xi_{std})$. We generalize their result to high-dimensions. 
\begin{corollary}\label{cor: contact_surgery_presentation}
	If $(Y^{2n-1}, \xi), n \ge 3,$ has an almost Weinstein filling, then $(Y^{2n-1}, \xi)$ can be obtained from $(S^{2n-1}, \xi_{std})$ by a sequence of contact surgeries and anti-surgeries. 
\end{corollary}
See Corollary \ref{cor: contact_surgery_presentation2}. 
If $(Y^{2n-1}, \xi)$ is obtained from $(S^{2n-1}, \xi_{std})$ by contact surgery or anti-surgery (of index $n$), then it has a smooth filling that admits a Morse function with critical points of index at most $n$. This is the same as having an almost Weinstein filling, except for the existence of an almost complex structure. Hence the condition that $(Y^{2n-1}, \xi)$ has an almost Weinstein filling cannot be significantly weakened. So our result is essentially sharp if one allows only contact surgery or anti-surgery of index $n$. If  \textit{coisotropic} contact surgeries of all indices are allowed, Conway and Etnyre \cite{conway_etnyre_caps} have informed us
 that all contact manifolds are attainable by surgery.
\begin{examples}
For any contact structure $(S^{2n-1}, \xi)$ in $(S^{2n-1}, J_{std})$, there exist Legendrians $\Lambda_1, \Lambda_2 \subset (S^{2n-1}, \xi_{std})$ such that $(S^{2n-1}, \xi)$ is obtained from $(S^{2n-1}, \xi_{std})$ by contact surgery on $\Lambda_1$ and contact anti-surgery on $\Lambda_2$. 
\end{examples}

One of the main problems in contact topology is to classify all \textit{convex} symplectic fillings of a given contact manifold $(Y, \xi)$.  
These are symplectic domains whose symplectic form expands outward near the boundary and induces the contact structure $(Y, \xi)$. Weinstein domains are a special case and the set of Weinstein fillings of $(Y, \xi)$ correspond to  morphisms in $\mathfrak{Contact}$ from the empty contact structure to $(Y, \xi)$. There has been much progress in understanding fillings in dimension 3; for example,  $S^3$  and $T^3$ with their standard contact structures have unique Weinstein fillings \cite{CE12, Wendl_planar_filling}. In high-dimensions, certain special  contact manifolds  also have very restricted fillings (at least restricted topologically). For example,  Eliashberg, Floer, and McDuff \cite{McD} showed that all exact fillings of the standard contact sphere $(S^{2n-1}, \xi_{std})$ are diffeomorphic to the ball; also see \cite{Geiges_subcritical, Lazarev_flexible_fillings} for generalizations to subcritically and flexibly-filled contact manifolds. On the other hand, certain contact manifolds have no fillings. Since there is no unital ring map from the zero ring to a non-zero ring, a contact manifold with vanishing contact homology cannot have a Weinstein (or exact) filling. Hence overtwisted structures have no fillings. 

There is also a long history of constructing contact manifolds with multiple fillings, especially in dimension 3.  Ozbagci and Stipcisz  \cite{OS} discovered the first example of a contact 3-manifold with infinitely many non-homotopy-equivalent Weinstein fillings. 
Controlling the topology of the fillings is quite subtle in this dimension but there are now examples of contact 3-manifolds with  infinitely many homeomorphic but not diffeomorphic fillings \cite{Smith_fillings}, fillings with arbitrary fundamental group \cite{Ozbagci_arbitrary_fundamental_gp}, 
``large" fillings with  unbounded Euler characteristic and signature \cite{large_fillings_Baykur_vanHornMorris}, and 
``small" fillings with $b_2 = 2$ \cite{small_fillings_Akbulut_Yasui}. The first example of a high-dimensional contact manifold with infinitely many non-homotopy equivalent Weinstein fillings is due to Oba \cite{Oba_infinite_fillings}. Many of these constructions use open book decompositions of contact manifolds and construct fillings by finding different factorizations of the open book monodromies into positive Dehn twists.
Such an approach is feasible in dimension 3 since the symplectic mapping class group agrees with the ordinary mapping class group for 2-dimensional surfaces and is generated by Dehn twists. The symplectic mapping class group of high-dimensional domains is much less understood; in general, it does not agree with the smooth mapping class group and is not generated by Dehn twists. 

We can use the maximal construction in Theorem \ref{thm: contact_cobordisms} to give an alternative construction of contact manifolds with many fillings that does not depend on understanding the high-dimensional symplectic mapping class group. 
Our construction converts Weinstein domains with \textit{almost} contactomorphic boundaries into domains with genuinely contactomorphic boundaries. 
\begin{corollary}\label{cor: contact_fillings}
	Let $W_1^{2n}, \cdots, W_k^{2n}, n \ge 3,$ be Weinstein domains such that $\partial W_i^{2n}$ are almost contactomorphic. Then there are Weinstein domains $X_i^{2n}$ that are almost symplectomorphic to $W_i^{2n}$ and contain $W_i^{2n}$ as a subdomain such that $\partial X_i^{2n}$ are  contactomorphic. 	 
	In particular, if $(Y^{2n-1}, J), n \ge 3,$ has an almost Weinstein filling $W^{2n}$, then for any $k \ge 0$, there is a contact structure $(Y, \xi_k)$ in $(Y, J)$ with at least $k$ non-homotopy-equivalent Weinstein fillings. 
\end{corollary}
\begin{remark}
	More generally,  $W_i^{2n}$ can be Liouville domains, in which case $X_i^{2n}$ will also be Liouville domains containing $W_i^{2n}$ as Liouville subdomains. 
\end{remark}
\begin{proof}
Since $(Y_i^{2n-1}, \xi_i): = \partial W_i^{2n}$ are almost contactomorphic, by Theorem \ref{thm: contact_cobordisms} there exist smoothly trivial Weinstein cobordisms $C_i^{2n}$ 
	such that $\partial_- C_i^{2n} = (Y_i^{2n-1}, \xi_i)$ and $\partial_+ C_i$ are all contactomorphic. Then  $X_i^{2n} := W_i^{2n} \circ C_i^{2n}$ satisfy all conditions in the first claim. For the second claim, it suffices find infinitely many non-homotopy-equivalent \textit{almost} Weinstein domains  $W_i^{2n}$ with boundary $(Y, J)$ and use Eliashberg's existence h-principle \cite{Eli90} for Weinstein domains. For example, let $W_i^{2n}: = W^{2n} \natural (\natural^i B^{2n}_0)$ where $B_0^{2n}$ is a certain Brieskorn manifold \cite{U} if $n$ is odd and  $B_0^{2n}$ is one of the manifolds from \cite{DG, Geiges} if $n$ is even. 
\end{proof}
Bowden, Crowley, and Stipsicz \cite{BCS} gave a topological criterion for an almost contact manifold to admit an almost Weinstein filling. Combining their result with Corollary \ref{cor: contact_fillings}, we get a topological criterion for an almost contact class to admit contact structures with arbitrarily many Weinstein fillings.
Furthermore, we have complete control over the smooth topology of the fillings and although we cannot control their symplectic topology, we can require our fillings to have prescribed subdomains since $W_i$ is a subdomain of $X_i$. 
Another  notable feature of our construction is that it uses flexible methods like the h-principle for overtwisted contact structures \cite{BEMtwisted} in contrast to the more algebraic approach of factorizing symplectomorphisms into Dehn twists 
\cite{OS, large_fillings_Baykur_vanHornMorris, Oba_infinite_fillings}. It would be interesting to relate these two approaches and determine whether our fillings define new relations in the symplectic mapping class group. In Section \ref{ssec: symplectomorphism}, we discuss some possible implications for the symplectic mapping class group from our maximal construction.

There is also an analog of Corollary \ref{cor: contact_fillings} in dimension 3 but it is necessarily weaker. This result follows from  Theorem \ref{thm: contact_cob_dimension4}, the 3-dimensional analog of Theorem \ref{thm: contact_cobordisms}. 
\begin{corollary}\label{cor: cont_fillings_dim3}
	Suppose $W_1^4, \cdots, W_k^4$ are Weinstein domains with almost contactomorphic contact boundaries. Then there exist Weinstein domains $X_1^4, \cdots, X_k^4$ with contactomorphic boundaries such that $X_i^4$ contains $W_i^4$ as a subdomain and $X_i^4 \backslash W_i^4$ has $k-1$ Weinstein 2-handles and no $0,1$-handles.
\end{corollary}

Corollary \ref{cor: contact_fillings} produces contact manifolds with many fillings. On the other hand,  certain special contact manifolds have few fillings. This contrast shows that these contact manifolds must be different, as must their bounding Weinstein domains.
\begin{examples}\label{ex:exotic_str_ball}
	For $n \ge 3$ and any $k \ge 1$, there exists a Weinstein domain $\Sigma^{2n}_k$ almost symplectomorphic to $B^{2n}$ such that its contact boundary $\partial \Sigma^{2n}_k$ has $k$ non-homotopy-equivalent Weinstein fillings. 
	By a result of Eliashberg-Floer-McDuff \cite{McD}, any Weinstein filling of the standard contact structure $(S^{2n-1}, \xi_{std}) = \partial B^{2n}_{std}$ must be diffeomorphic to $B^{2n}$. Hence the contact boundary $\partial \Sigma_k$ is not contactomorphic to $(S^{2n-1}, \xi_{std})$, which implies that $\Sigma^{2n}_k$ is not Weinstein homotopic to $B^{2n}_{std}$ and the completion $\widehat{\Sigma_k}$  is not symplectomorphic to $(\mathbb{C}^n, \omega_{std})$.	There are many existing constructions of exotic Weinstein balls \cite{Seidel_Smith_ramanujam_surface, MM, abouzaid_seidel_recombination}, distinguished by the presence of non-displaceable Lagrangian tori or symplectic cohomology. 
	 However it is unknown whether these previous examples have exotic Weinstein fillings of their contact boundary; our domains $\Sigma_k^{2n}$ are built precisely to have such fillings. 
\end{examples}

Now we discuss some applications of the maximal construction to 
\textit{convex} symplectic fillings, also called
\textit{symplectic caps}. For caps, the symplectic form expands inward near the boundary and so the symplectic structure cannot be exact by Stoke's theorem. 
There is no topological obstruction to the existence of a symplectic cap and any almost contact manifold has an \textit{almost} symplectic cap  \cite{BCS}; for the same reason, there is no topological obstruction to a convex symplectic filling although we know that these do not always exist. Lisca and Mati{\'{c}} \cite{Lisca1997} showed that any contact manifold with a Weinstein filling has a cap. Using this result along with the fact that the mapping class group of surfaces is generated by Dehn twists, Etnyre and Honda \cite{Etnyre_Honda_caps} showed that all contact 3-manifolds have symplectic caps; Eliashberg  \cite{Eliashberg_caps} gave a different proof of this result.
The existence of these caps was a crucial ingredient in the proof of Property P for knots \cite{Kronheimer_Mrowka_PropP}. 
Later Eliashberg and Murphy \cite{EM} showed that overtwisted contact manifolds in any dimension admit symplectic caps. Hence concave fillings seem more flexible than their convex siblings and there are many contact manifolds that have no convex fillings but do have concave fillings. 
However once certain topological conditions are imposed on the cap, they become quite rigid and, in fact, a useful tool for classifying convex fillings \cite{caps_Li_Mak_Yasui, McD, Etnyre_planarOBD}.
Therefore symplectic caps seemed to be on the boundary of symplectic rigidity and flexibility and it was unclear how restrictive caps are. 
Wendl \cite{Wendl_blog3} asked whether all contact structures admit symplectic caps. In the following result, we will show that this is indeed the case. 
\begin{corollary}\label{cor: all_caps}
	Every contact manifold $(Y^{2n-1}, \xi), n \ge 3,$ has a symplectic cap. 
\end{corollary}
\begin{proof}
By applying Theorem \ref{thm: contact_cobordisms} to $(Y^{2n-1}, \xi)$ and any Weinstein-fillable contact manifold $\partial W^{2n}$, we see that there is a Weinstein cobordism $C^{2n}_0$  from $(Y^{2n-1}, \xi)$ to a different Weinstein-fillable contact structure $\partial V^{2n}$;  see Corollary \ref{cor: cobordant_to_fillable} for details. 
Lisca and Mati{\'{c}} \cite{Lisca1997} proved that $\partial V^{2n}$ has a symplectic cap $C^{2n}$. Then $C^{2n}_0 \circ C^{2n}$ is a symplectic cap of $(Y^{2n-1}, \xi)$ as desired. 
\end{proof}
While completing this paper, we learned that Conway and Etnyre \cite{conway_etnyre_caps} have proven a similar result. Combining Corollary \ref{cor: all_caps} with the existing proof of the $n = 2$ case \cite{Etnyre_Honda_caps, Eliashberg_caps}, we see that all contact manifolds, in any dimension, have symplectic caps. 
Corollary \ref{cor: all_caps} does not seem to provide any control over the topology of the symplectic cap of a given contact manifold. This is because the crucial result \cite{Lisca1997} also does not provide any control. It is possible that symplectic caps with wildly different topology are needed to cap off almost contactomorphic contact structures. In fact, by gluing Weinstein cobordisms on top of $\partial(W^{2n} \circ C_2^{2n})$ with arbitrarily large middle-dimensional homology with positive intersection form and capping them off, we can produce infinitely many symplectic caps of the same contact manifold with different topology. Hence symplectic caps of contact manifolds are not unique. We do not know how to determine the smallest symplectic cap of a given contact manifold, a question which seems related to the open problem of existence of symplectic structures on closed manifolds.

Corollary \ref{cor: contact_maximal_weinstein} shows that for any finite collection of contact manifolds, possibly with different topology, there is a maximal element with respect to Weinstein cobordism. We do not know whether there is a contact structure that is maximal for \textit{all} contact manifolds. A  weaker notion than Weinstein cobordism is that of a \textit{strong symplectic cobordism}, which is exact near the boundary but perhaps not in the interior. 
Wendl \cite{Wendl_blog7} showed that symplectic caps are quite useful for constructing strong symplectic cobordisms between contact manifolds; for example, he showed that in dimension 3, there is a contact structure that is maximal for \textit{all} contact manifolds with respect to strong cobordisms. Combining Wendl's argument with the existence of caps in Corollary \ref{cor: all_caps}, we can prove that there exists such a maximal contact structure in high-dimension dimensions. 
\begin{corollary}\label{cor: maximal_strong_cobordism}
	For $n \ge 3$, there exists a connected, non-empty maximal contact manifold $(Y_{max}^{2n-1}, \xi_{max})$ such that for any  $(Y^{2n-1}, \xi)$, possibly with different topology, there exists a strong symplectic cobordism $W^{2n}$ with $\partial_- W = (Y,\xi)$ and $\partial_+ W = (Y_{max}, \xi_{max})$.
\end{corollary}
\begin{remark}
	The smooth and symplectic topology of $W^{2n}$ of course depend on $(Y, \xi)$. The key point is that the contact structure on $\partial_+W^{2n}$ is independent of $(Y, \xi)$. 	
\end{remark}
\begin{proof}
	Following Wendl's argument \cite{Wendl_blog7}, let $W^{2n}$ be a Liouville domain such that $\partial W$ is disconnected with two components $(Y_1, \xi_1), (Y_2, \xi_2)$; such domains exist in all dimensions by \cite{Massot_Niederkruger_Wendl_weak_filling}. There is a Weinstein cobordism $W_1$ from $(Y_2, \xi_2) \coprod (Y^{2n-1}, \xi)$ to $(Y_2, \xi_2) \sharp (Y^{2n-1}, \xi)$ given by a Weinstein 1-handle. By Corollary \ref{cor: all_caps}, $(Y_2, \xi_2) \sharp (Y^{2n-1}, \xi)$ has a symplectic cap $C^{2n}$. Then we can glue $W^{2n}, W_1^{2n}$ along $(Y_2, \xi_2)$ and $W_1^{2n}, C^{2n}$ along $(Y_2, \xi_2) \sharp (Y^{2n-1}, \xi)$ to get a strong symplectic cobordism from $(Y, \xi)$ to $(Y_1, \xi_1)$ as desired. Hence we can set $(Y_{max}, \xi_{max}): =  (Y_1, \xi_1)$.
\end{proof}
The proof above shows that the maximal contact structure is in fact quite explicit: it can be taken to be any component of any Liouville domain with disconnected boundary. However this maximal contact structure is not unique. Any other contact structure obtained from this structure via contact surgery is also maximal. 

Our final application of the maximal construction Theorem \ref{thm: contact_cobordisms} is to isocontact embeddings. Gromov \cite{gromov_hprinciple} proved an h-principle for 
isocontact embeddings of codimension at least 4: if $(Y^{2m+1}, \xi)$ admits an almost contact embedding into $(Z^{2n+1}, \xi)$ and $m\le n-2$, then there is a genuine contact embedding of $(Y^{2m+1}, \xi)$ into  $(Z^{2n+1}, \xi)$. Recently,  Pancholi and Pandit \cite{pancholi_contact_embeddings} used open book decompositions and overtwisted contact structures to prove an h-principle type result in the codimension two case; also see \cite{etnyre_fukukawa_embeddings, etnyre_lekili_embedding} for embeddings of contact 3-manifolds into contact 5-manifolds via braided and spun embeddings. We will give an alternative proof of the result of Pancholi and Pandit using a variation of our maximal construction and Murphy's existence h-principle for loose Legendrians \cite{Murphy11}. 
\begin{theorem}\label{thm: codimension_2}
	If $(Y^{2n-1}, \xi_0), n \ge 3,$ has a contact embedding into $(Z^{2n+1}, \xi)$ with trivial normal bundle and $(Y^{2n-1}, \xi_1)$ is almost contactomorphic to $(Y^{2n-1}, \xi_0)$, then $(Y^{2n-1}, \xi_1)$ also has a contact embedding into $(Z^{2n+1}, \xi)$.
\end{theorem}
See Theorem \ref{thm: codimension_22}. We do not know whether Theorem \ref{thm: codimension_2} holds under the weaker assumption that $(Y^{2n-1}, \xi_1)$ just has an almost contact embedding into $(Z^{2n+1}, \xi)$, without the existence of a contact embedding of $(Y^{2n-1}, \xi_0)$. If  $(Z^{2n+1}, \xi_{ot})$ is overtwisted, then our result does hold under this weaker assumption, as proven by Borman, Eliashberg, and Murphy \cite{BEMtwisted}. We also note that Casals, Murphy, and Presas \cite{CMP} have shown that  a sufficiently large neighborhood of an overtwisted contact submanifold is also overtwisted. On the other hand, Theorem \ref{thm: codimension_2} shows that there are many codimension two embeddings of overtwisted contact manifolds into tight contact manifolds. Therefore the neighborhoods of the overtwisted submanifolds in the ambient contact manifold must be quite small; indeed, a sufficiently small neighorhood of any codimension two contact submanifold is tight  \cite{presas_tight_neighborhood}, even if the contact submanifold is abstractly overtwisted. 
\\

Now we give an outline of the rest of the paper. In Section \ref{sec: background}, we provide some background material on symplectic cobordisms and introduce the stacking construction. In Section \ref{sec: maximal_weinstein_domains}, we apply this construction to Weinstein domains and prove the results stated in Section \ref{ssec: cat_weinstein}. In Section \ref{sec: maximal_contact}, 
we construct maximal contact structures and prove the results from Section \ref{ssec: cat_contact}. In Section \ref{sec: Lag_fillings}, we consider some applications to Lagrangians.

\section*{Acknowledgements} 
We thank  Mohammed Abouzaid, Roger Casals, Emmy Murphy, and Zach Sylvan for helpful discussions. This work was partially supported by a NSF postdoc fellowship.

\section{Background}\label{sec: background}

In this section, we present some background material on symplectic cobordisms and discuss several gluing constructions. 

\subsection{Liouville and Weinstein cobordisms}
Before defining Weinstein domains and cobordisms, we first review a more general type of symplectic cobordism. 
A \textit{Liouville} cobordism $(W, \lambda )$ is a smooth cobordism $W$
with boundary $\partial W = \partial_- W \coprod \partial_+ W$ 
that is  equipped with a 1-form $\lambda$ such that $d\lambda$ is a symplectic form; in addition, the Liouville vector field $X$ defined by $i_X d\lambda = \lambda$ must be inward, outward transverse to the boundaries $\partial_-W, \partial_+W$ respectively. In this case, $(\partial_- W; ker \lambda|_{\partial_-W} ), (\partial_+ W; ker \lambda|_{\partial_+W} )$ are contact structures. We say that $(W, \lambda )$ is a Liouville cobordism between $(Y_-, \xi_-), (Y_+, \xi_+)$ if $\partial_-W, \partial_+W$ with the induced contact structures are contactomorphic to 
$(Y_-, \xi_-), (Y_+, \xi_+)$ respectively. A Liouville \textit{domain} is a Liouville cobordism with $\partial_-W = \emptyset$.

The natural notion of equivalence of Liouville cobordisms or domains is a \textit{Liouville homotopy}, a deformation through Liouville structures. As shown in \cite{CE12}, two homotopic Liouville domain $W_1, W_2$ have exact symplectomorphic \textit{completions} $\widehat{W}_1, \widehat{W}_2$, which are the open symplectic manifolds obtained by gluing the domains to the symplectizations of their contact boundaries. Homotopic domains also have contactomorphic boundaries.
These results demonstrates the importance of Liouville domains and homotopies 
since Moser's trick does not generally hold for open manifolds; for example, any two symplectic structures on $\mathbb{R}^{2n}$ are isotopic by Gromov's h-principle \cite{gromov_hprinciple}, i.e. can be connected through symplectic structures, but not all are symplectomorphic \cite{MM}.

A \textit{Weinstein cobordism} $(W^{2n}, \lambda, \phi)$ is a Liouville cobordism that admits a Morse function $\phi: W \rightarrow \mathbb{R}$  compatible with the Liouville structure. More precisely, $\phi$ is constant on $\partial_{\pm} W$  and the Liouville vector field $X$ is gradient-like for $\phi$. This implies that the stable manifolds of $X$ are isotropic with respect to the symplectic form $d\lambda$ \cite{CE12} and hence the critical points of $\phi$ have index at most $n$. Therefore admitting a Weinstein structure severely restricts the topology of $W^{2n}$. McDuff \cite{McD} constructed Liouville domains which do not satisfy these topological conditions and therefore are not Weinstein. However there are no known examples of Liouville domains that satisfy these topological conditions but are not Weinstein.

Associated to the Weinstein Morse function $\phi: W^{2n} \rightarrow \mathbb{R}$ is a natural collection of contact submanifolds and Weinstein domains inside $W$.
Namely, the regular level sets $\phi^{-1}(c)$ have natural contact structures with contact form $\lambda|_{\phi^{-1}(c)}$ 
and the sublevel sets $\phi^{-1}(\le c)$ 
are Weinstein \textit{subdomains} of $W$. 
These subdomains change in a precise way when we pass through a critical value of $\phi$.
Suppose that $p$ is a critical point of $\phi$ of index $k$ with critical value $\phi(p) = c$.
Let   
$W_- = \phi^{-1}(\le c-\epsilon), 
W_+ = \phi^{-1}(\le +\epsilon)$ be  Weinstein subdomains
below, above the critical value respectively and let 
$(Y_-, \xi_- ) = \partial W_-, 
(Y_+, \xi_+) = \partial W_+$
be nearby regular level sets with their induced contact structures.
The $X$-stable manifold of $p$ is an isotropic $k-$disk and intersects $Y_-$ is an isotropic $(k-1)-$sphere $\Lambda_p$ (isotropic with respect to the contact structure).  
Then $W_+$ is obtained from $W_-$ by attaching  a \textit{Weinstein handle} $H^k$ to $W_-$ along $\Lambda_p \subset (Y_-,\xi_-) = \partial W_-$, the \textit{attaching sphere} of the handle.
This procedure is constructive and any Weinstein domain can be built up this way; hence a Weinstein Morse function gives a symplectic handle-body decomposition of $W$.
We will use the notation 
$W^{2n} \cup H^k_\Lambda$ to denote the Weinstein domain obtained by 
Weinstein handle attachment along a (framed) isotropic sphere $\Lambda^{k-1} \subset \partial W$. 

We also note that
Weinstein handle attachment  changes the contact boundary $(Y^{2n-1}, \xi) = \partial W^{2n}$. The operation taking  $\partial W^{2n}$ to $\partial(W^{2n} \cup H^k_\Lambda)$ is called \textit{contact surgery} and it makes sense for any contact manifolds, not just the boundaries of Weinstein domains. 
We will use the notation
$(Y^{2n-1}, \xi) \cup H^k_{\Lambda}$ to denote the contact surgery of $(Y, \xi)$ along $\Lambda \subset (Y,\xi)$. 
Contact $n$-surgery  has an inverse operation called contact \textit{anti-surgery}. Given a Legendrian sphere $\Lambda_+ \subset (Y_+, \xi_+)$, there is a contact manifold $(Y_-, \xi_-)$ and a Legendrian $\Lambda_- \subset (Y_-, \xi_-)$
such that $(Y_+, \xi_+) = (Y_-, \xi_-) \cup H^n_{\Lambda_-}$ and the \textit{belt} sphere of $H^n_{\Lambda_-}$ coincides with $\Lambda_+ \subset (Y_+, \xi_+)$. We  use the notation   $(Y_+, \xi_+) \cup H^n_{\Lambda}[+1]$ to denote the anti-surgered contact manifold
$(Y_-, \xi_-)$.

The natural notion of equivalence between Weinstein structures
$(W, \lambda_0, \phi_0), (W, \lambda_1, \phi_1)$ on a fixed manifold $W$ is a \textit{Weinstein homotopy}. This is an interpolating 1-parameter family of structures $(W, \lambda_t, \phi_t), t \in [0, 1],$ that are Weinstein except at isolated $t$ at which $\phi_t$ has a
birth-death singularity,  and are Liouville for all $t$. 
From the handlebody point of view, Weinstein homotopies correspond to a sequence of three moves: isotopies of the attaching spheres
through isotropics, changing the order of attachment of handles that are not connected by gradient trajectories, and handle-slides. Since a Weinstein homotopy is a case of Liouville homotopy, homotopic Weinstein structures have symplectomorphic completions. 
Therefore we can consider Weinstein domains and associated objects up to Weinstein homotopy. For example, when we say that $W_0 = (W_0, \lambda_0, \phi_0)$ is a Weinstein subdomain of $W_1 = (W_1, \lambda_1, \phi_1)$, we mean that there is a Weinstein homotopy of $(W_1, \lambda_1, \phi_1)$ to  $(W_1, \lambda_1', \phi_1')$ so that $(W_0, \lambda_0, \phi_0)$ is a sublevel set of $\phi_1'$. 
For a Weinstein subdomain $W_0$ of $W_1 = (W_1, \lambda_1, \phi_1)$ that already is a sublevel set of $\phi_1$, a Weinstein homotopy of \textit{pairs} $(W_1, W_0)$ is a Weinstein homotopy of $W_1$ such that the subdomain $W_0$ is preserved by the homotopy, i.e $W_0$ is a sublevel set of the entire family of Weinstein Morse functions $\phi_t$. In particular, this is a Weinstein homotopy of the cobordism $W_1 \backslash W_0$. If $(W_1, W_0), (W_2, W_0)$ are Weinstein homotopic pairs, then there is an exact symplectomorphism $f: \widehat{W}_1 \rightarrow \widehat{W}_2$ that is the identity between $W_0 \subset  W_1$ and $W_0 \subset W_2$. 

In this paper, we will be mainly concerned with two special types of Weinstein cobordisms. A Weinstein cobordism $W^{2n}$ is \textit{subcritical} if it admits a Weinstein Morse function $\phi$ whose critical points have index \textit{strictly} less than $n$. A Weinstein cobordism $W^{2n}$ is \textit{flexible} if its index $n$ critical points are attached along \textit{loose} Legendrian spheres; see \cite{Murphy11, CE12} for details.  In particular, any subcritical cobordism is flexible. 
Loose Legendrians have dimension at least $2$ and hence flexible domains are defined only for $n \ge 3$. These two types of cobordisms satisfy existence and uniqueness h-principles \cite{CE12}. We say that a smooth cobordism $W^{2n}$ is \textit{almost Weinstein} if it admits an almost complex structure and a Morse function $\phi$ that is constant on $\partial W = \partial_- W \coprod \partial_+W$ and has critical points of index at most $n$. The existence h-principle \cite{Eli90} states that any almost Weinstein domain $W^{2n}, n \ge 3$, admits a flexible Weinstein structure. The uniqueness h-principle \cite{CE12} states that any two flexible Weinstein domains that are almost symplectomorphic, i.e. are connected by non-degenerate 2-forms, are Weinstein homotopic (and hence have symplectomorphic completions).

\subsection{Concatenating and stacking} \label{ssec: concatenate_stacking}
Now we discuss gluing Liouville and Weinstein cobordisms. Liouville cobordisms have a standard form near their boundaries. If $\partial_- W = (Y_-, \xi_-)$, there is an open neighborhood of $\partial_-W$ of the form $([0, \epsilon) \times Y_-, \lambda = e^t \alpha, X= \frac{\partial}{\partial t})$; here $t$ is the coordinate on the $[0, \epsilon)$ factor and $\alpha$ is some contact form for $(Y_-, \xi_-)$. A similar statement holds for $\partial_+W$. 
Using these collar neighborhoods, it is possible to glue Liouville cobordisms. More precisely, if $W_1, W_2$ are Liouville cobordisms such that $\partial_+W_1$ is contactomorphic to $\partial_-W_2$, then there exists a Liouville cobordism obtained by gluing $W_1,W_2$ along collared neighborhoods of $\partial_+W_1, \partial_-W_2$. We will denote the resulting cobordism $W_1 \circ W_2$ and call it the \textit{concatenation} of $W_1, W_2$. It is independent of the choice of collared neighborhoods up to Liouville homotopy. However $W_1 \circ W_2$ does depend on the choice of contactomorphism between $\partial_+ W_1, \partial_- W_2$; to make concatenation well-defined, we will concatenate $W_1, W_2$ only when $\partial_+W_1, \partial_-W_2$ actually coincide. 
Note that
$\partial_-(W_1 \circ W_2) = \partial_-W_1, \partial_+(W_1 \circ W_2) = \partial_+W_2$, and $W_1 \circ W_2$ contains $W_1$ as a Liouville
subcobordism. If $W_1,W_2$ are Weinstein cobordisms, then $W_1 \circ W_2$ is also a Weinstein cobordism.
As explained in Sections \ref{ssec: cat_weinstein},  the morphisms in the categories $\mathfrak{Weinstein}(W,J), \mathfrak{Contact}(Y,J)$ are composed via concatenation. 

The concatenation construction requires consecutive positive and negative contact boundaries of Liouville cobordisms to agree. In this paper, we introduce a different gluing construction of Weinstein cobordisms called \textit{stacking} that only requires the \textit{negative} ends of Weinstein cobordisms to agree. The following result implies Theorem \ref{thm: stacking} from the Introduction, which we phrased as a certain extension property in  $\mathfrak{Weinstein}(W,J)$. 
\begin{theorem}\label{thm: stacking2}
If  $W_1^{2n}, \cdots, W_k^{2n}$ are Weinstein cobordisms such that $\partial_- W_i=(Y, \xi)$ for all $i$, then there exists Weinstein cobordisms $C_i$ with $\partial_- C_i = \partial_+ W_i$  such that $W_i \circ C_i$ are all Weinstein homotopic to a Weinstein cobordism $W$; if $W_i^{2n}$ are smoothly trivial and $n \ge 3$,  so are $C_i^{2n}$ and $W^{2n}$. Furthermore, $W^{2n}$ is covered by the subdomains $W_i^{2n}$ and is homotopy-equivalent, in the sense of topological spaces, to $\coprod W_i / (\partial_- W_i \sim \partial_- W_j)$.
\end{theorem}
\begin{proof}
Let $S_i \subset W^{2n}_i$ be the skeleton of Liouville vector field $X_i$ of the Weinstein structure  $(W^{2n}_i, \lambda_i, \phi_i)$, i.e. the set of points of $W^{2n}$ that does not reach to $\partial_+W^{2n}$ under the $X_i$-flow. As shown in \cite{CE12}, this singular space is compact and stratified by isotropic disks, namely the cores of the Weinstein handles of $W^{2n}_i$. The intersection $\Lambda_i = S_i \cap (Y^{2n-1}, \xi)$ is similarly stratified by isotropic submanifolds, which have dimension at most $n-1$; this is the image of the attaching spheres of the Weinstein handles of $W^{2n}_i$ in $(Y^{2n-1}, \xi)$. We note that $S_i$ and $\Lambda_i$ may each have many components. 

We now show that there are ambient contactomorphisms $\phi_i$ of $(Y^{2n-1}, \xi)$ contact isotopic to the identity such that $\phi_i(\Lambda_i)$ are all disjoint. First we suppose that $\Lambda_1$ is a connected smooth Legendrian. Then a small neighborhood of $\Lambda_1$ is contactomorphic to $J^1(\Lambda_1)$ and nearby Legendrians are given by graphs of 1-jets of functions. Thom's jet transversality theorem shows that for any stratified submanifold $\Sigma$ of $J^1(\Lambda_1)$ (not necessarily isotropic) whose top-dimensional smooth strata have dimension $k < n$, there exists a $C^0$-small function $f: \Lambda_1 \rightarrow \mathbb{R}$ whose 1-jet $\Gamma(f)$ in
$J^1(\Lambda_1)$ is disjoint from $\Sigma$; see Theorem 2.3.2 of \cite{eliashberg_mishachev}. Since $\Lambda_2, \cdots, \Lambda_k$ are stratified by isotropics, which have dimension less than $n$, we can apply Thom's theorem to
$\Sigma = (\Lambda_2 \cup \cdots \cup \Lambda_k )\cap J^1(\Lambda_1)$ and conclude that there exists a function $f: \Lambda_1 \rightarrow \mathbb{R}$ such 
$\Gamma(f) \subset J^1(\Lambda_1)$ is disjoint from $(\Lambda_2 \cup \cdots \cup \Lambda_k )\cap J^1(\Lambda_1)$. Since $\Lambda_1$ and $\Gamma(f)$ are Legendrian isotopic in $J^1(\Lambda_1)$, there exists an ambient contactomorphism $\phi_1$ contact isotopic to the identity such that $\phi_1(\Lambda_1) = \Gamma(f)$ is disjoint from $\Lambda_2, \cdots, \Lambda_k$. If $\Lambda_1$ is not smooth (or disconnected), then we construct $\phi_1$ by induction on the strata and components of $\Lambda_1$ 
(for example, by using thickenings of the subcritical strata to make them Legendrian). Finally, we construct $\phi_2, \cdots, \phi_k$ by induction. 

Now we attach Weinstein handles along \textit{all} the  $\Lambda:= \coprod_{i=1}^k \phi_i(\Lambda_i) \subset (Y, \xi)$. This is possible the $\phi_i(\Lambda_i)$ are disjoint and handle attachment changes the contact manifold only in a small neighborhood of the attaching sphere; therefore, the rest of isotropic spheres persist to the new contact manifold when we attach a handle. 
If we can first attach along $\phi_i(\Lambda_i)$, the resulting cobordism is Weinstein homotopic to $W_i$ since $\phi_i(\Lambda_i)$ is contact isotopic to $\Lambda_i$, the attaching spheres of $W_i$. Then we view $\coprod_{j \ne i} \phi_j(\Lambda_j)$ as an isotropic subspace of $\partial_+ W_i$ and let $C_i$ be the Weinstein cobordism with $\partial_-C_i = \partial_+ W_i$ obtained by attaching along $\coprod_{j \ne i} \phi_j(\Lambda_j)$. Since the order of handle attachment amongst the different $\phi_i(\Lambda_i)$ does not matter, all the $W_i \circ C_i$ are Weinstein homotopic and we call this common Weinstein cobordism $W$. 
Note that the $W_i$'s cover $W$ since $W$ is just the union of all the $W_i$'s, glued along their common subspace $(Y, \xi) = \partial_-W_i$; in particular, $W$ is homotopy-equivalent,  in the sense of topological spaces, to $\coprod W_i / (\partial_-W_i \sim \partial_- W_j)$. 

Finally, suppose that $W_i^{2n}$ are all smoothly trivial and $n \ge 3$. For simplicity, we will assume that $W_i^{2n} = H^{n-1}_i \cup H^n_{\Lambda_i}$ for a single Legendrian $\Lambda_i$ that can be smoothly isotoped to intersect the belt sphere of $H^{n-1}_i$ once; this can always be assumed to be the case \cite{Lazarev_critical_points}. Then $W = H^{n-1}_1 \cup \cdots H^{n-1}_k \cup H^n_{\Lambda_1} \cdots \cup  H^n_{\Lambda_k}$
and $C_i = W \backslash H^{n-1}_i \cup H^n_{\Lambda_i}$. 
By looking at the trace of this smooth isotopy, we see that there are Whitney disks that cancel out all intersection points of $\Lambda_i$ with this belt sphere (except for one). Since $n \ge 3$, these Whitney disks are generically disjoint from any other Legendrian sphere $\Lambda_j$. Hence we can use these disjoint Whitney disks to smoothly isotope the link $\Lambda_1 \coprod \cdots \coprod \Lambda_k \subset (Y, \xi) \cup H^{n-1} \cup \cdots \cup H^{n-1}_k$  to a link 
$\Lambda_1' \coprod \cdots \coprod \Lambda_k'$
such that each $\Lambda_i'$ intersects the belt sphere of $H^{n-1}_i$ exactly once. Therefore $W^{2n}$ and $C_i^{2n}$ are also   smoothly trivial.
\end{proof}

We will use the following notation to denote a Weinstein cobordism constructed as in Theorem \ref{thm: stacking2}: 
$$
Stack_{(Y, \xi)}(W_1, \cdots, W_k)
$$
We include the common contact manifold $(Y, \xi)$ in the notation to highlight that the gluing is done along $(Y, \xi)$. As seen in the proof of Theorem \ref{thm: stacking2}, the construction of $Stack_{(Y, \xi)}(W_1, \cdots, W_k)$ depends on choices and hence this cobordism is not  well-defined  in terms of just $W_1, \cdots, W_k$; see the discussion below. 
So the notation $Stack_{(Y, \xi)}(W_1, \cdots, W_k)$  just refers to a general cobordism constructed as in Theorem \ref{thm: stacking2}. 

We will also consider a slight generalization of 
$Stack_{(Y, \xi)}(W_1, \cdots, W_k)$. If $W_1, \cdots, W_k$ are Weinstein cobordisms with a common Weinstein subcobordism $\phi_i: W_0 \hookrightarrow W_i$, then we let 
$$
Stack_{W_0}(W_1, \cdots, W_k)
$$
denote the Weinstein cobordism obtained by gluing along the common subcobordism $W_0$; more precisely, $Stack_{W_0}(W_1, \cdots, W_k) := W_0 \circ Stack_{\partial W_0}(W_1\backslash \phi_1(W_0), \cdots, W_k\backslash \phi_k(W_0))$.
We will write $Stack_{W_0}((W_1, \phi_1), \cdots, (W_k, \phi_k))$ when we want to emphasize that the embeddings $\phi_i$ are part of the data. 
This Weinstein cobordism contains $W_i$ and $W_0$ as subcobordisms such that the inclusion of $W_0$ into $Stack_{W_0}(W_1, \cdots, W_k)$ factors through the inclusion of $W_i$ into $Stack_{W_0}(W_1, \cdots, W_k)$. 
If $W_0$ is the trivial Weinstein cobordism 
$(Y, \xi)\times [0, 1]$, 
then we recover the previous construction
$Stack_{(Y, \xi)}(W_1, \cdots, W_k)$.
We also note that this gluing can be done along more general objects like Weinstein sectors or Liouville cobordisms; it is important that the complements $W\backslash W_0$ are Weinstein cobordisms
but the objects we glue along can be Liouville. 

Theorem \ref{thm: stacking2} shows that the Weinstein cobordisms $W_i \circ C_i$ are Weinstein homotopic for different $i$. These Weinstein homotopies are quite special because they have constant Liouville vector field. More precisely, we fix some constants $a < b < c$. Then there is a fixed Liouville vector field $X$ on $W$ and $k$ Morse functions $\phi_i: W^{2n} \rightarrow \mathbb{R}$ such that $X$ is gradient-like for all $\phi_i$ and 
$\phi_i^{-1}(a) = \partial_- W = (Y, \xi), \phi_i^{-1}(c) = \partial_+ W, \phi_i^{-1}([a,b]) = W_i$ and
$\phi_i^{-1}([b,c]) = C_i$. 
Furthermore, if we fix Weinstein functions  $\psi_i : W_i \rightarrow \mathbb{R}$ such that $\psi_i^{-1}(a)  = \partial_-W_i, \psi_i^{-1}(b) = \partial_+W_i$, then we can assume that 
$\phi_i|_{W_i} =  \psi_i$, i.e. $\phi_i$ extends $\psi_i$. 
Since $X$ is gradient-like for all $\phi_i$, the
Weinstein homotopy from $(W, X, \phi_i)$ to $(W, X, \phi_j)$ can be given by $(W, X, (1 - t)\phi_i + t\phi_j )$; this will be a Weinstein homotopy if $\phi_i, \phi_j$ are generic. The Liouville vector field $X$ is independent of $t$ and hence these Weinstein structures have the same Liouville skeleton.

From the proof of Theorem \ref{thm: stacking2}, it is clear that we can stack Weinstein cobordisms so that $Stack_{(Y, \xi)}(W_1, \cdots, W_k)$ is Weinstein homotopic to $Stack_{(Y, \xi)}(W_{\sigma(1)}, \cdots, W_{\sigma(k)})$ where $\sigma$ is any permutation of $\{1, \cdots, k\}$. 
Also, we can also perform the construction so that 
$Stack_{(Y, \xi)}(W_1, \cdots, W_i, Stack_{(Y, \xi)}(W_{i+1}, \cdots, W_{k}))$ and  
$Stack_{(Y, \xi)}(W_1, \cdots, W_i, W_{i+1}, \cdots, W_k)$ are Weinstein homotopic. However, as pointed out above, there are many choices when constructing $Stack_{(Y, \xi)}(W_1, \cdots, W_k)$ from $W_1, \cdots, W_k$; indeed many $W$ and $C_i$ satisfy the conditions in Theorem \ref{thm: stacking2}.
As defined, $W$ depends on the Weinstein presentation of $W_1, \cdots, W_k$. Given Weinstein presentations of $W_i$, there are still further choices to be made
by viewing the individual attaching spheres of the $W_i$ as a \textit{link}.  These choices can even affect the smooth topology of $Stack_{(Y,\xi)}(W_1, \cdots, W_k)$, as the following examples demonstrate.
\begin{examples}\label{ex: stacking_diff_topology}
Let $W_1 = T^*S^n = B^{2n}\cup H^n_{\Lambda_{unknot, 1}}$ and $W_2 = T^*S^n = B^{2n}\cup H^n_{\Lambda_{unknot, 1}}$. To construct $Stack_{B^{2n}}(W_1, W_2)$ we need to view $\Lambda_{unknot, 1}\coprod \Lambda_{unknot,2}$ as a Legendrian link in the \textit{same} contact manifold $(S^{2n-1}, \xi_{std})= \partial B^{2n}$. In general, there are many different ways to do this, even smoothly. For example, if $\Lambda_{unknot, 1}\coprod \Lambda_{unknot,2} \subset (S^{2n-1}, \xi_{std})$ are Legendrian unlinked, then $Stack_{B^{2n}}(W_1, W_2) = T^*S^n \natural T^*S^n$, the boundary connected sum of two copies of $T^*S^n$. On the other hand, if $\Lambda_{unknot,2}$ is a Reeb pushoff of $\Lambda_{unknot,1}$ (so that they have linking  number $-1$), then $Stack_{B^{2n}}(W_1, W_2) = T^*S^n \sharp_p T^*S^n$ is the plumbing of two copies of $T^*S^n$. So these two constructions yield manifolds with different intersection forms. 
\end{examples}

\begin{examples}
Generalizing the previous example, we note that Weinstein handle attachment is the special case of stacking when the relevant Weinstein cobordisms have only one critical point each. Consider two Legendrian spheres $\Lambda_1, \Lambda_2 \subset (Y, \xi)$ and Weinstein cobordisms $W_i = (Y, \xi) \times [0, 1] \cup H^n_{\Lambda_i}, i = 1, 2,$  obtained by attaching a Weinstein handle to $(Y, \xi)$ along $\Lambda_i$. These cobordisms have common negative boundary $(Y,\xi)$ and so we can construct $Stack_{(Y, \xi)}(W_1, W_2)$. This is precisely the result of Weinstein handle attachment to some Legendrian link 
$\Lambda_1 \coprod \Lambda_2$ such that the individual components are
Legendrian isotopic to $\Lambda_1, \Lambda_2$ respectively.
So in this case the stacking operation is not well-defined because a choice of two Legendrian embedding (up to Legendrian isotopy of each component) does not determine a Legendrian link. Even if we fix the two Legendrian embeddings, the two Legendrians might intersect and we will need to perturb them to get an embedded link; this choice of perturbation can lead to non-isotopic Legendrian links. 
\end{examples}

\begin{examples}\label{example: flexible_stacking}
As we will see in the proof of Proposition \ref{prop: all_domains_realized}, for any two Weinstein domains $W, V$ and any way of constructing $Stack(W, V)$, we can construct $Stack(W_{flex}, V_{flex})$ so that it is Weinstein homotopy equivalent to  $Stack(W, V)$. In particular, we can arrange so that all the symplectic data of  $Stack(W, V)$ is contained in the linking of the loose Legendrian attaching spheres of $W_{flex}, V_{flex}$ and hence is not uniquely defined just from the data of $W_{flex}, V_{flex}$. In particular, any Weinstein structure $W$ is Weinstein homotopic to $Stack_{W_{flex}}(W_{flex}, W_{flex})$. 
\end{examples}

Of course if $W_i$ are Liouville cobordisms such that $\partial_+W_i = \partial W_{i+1}$, then we can concatenate these cobordisms to produce the cobordism $W := W_1 \circ W_2 \circ \cdots \cdot W_k$. However even in
this restricted case, it is not clear that we can Liouville homotope $W$ to $W_i \circ C_i$ for some Liouville cobordism $C_i$  so that 
$W_i$ is the lower level cobordism. So in the Liouville case, it is not clear that we can switch the order of the cobordisms arbitrarily. We also note that Theorem \ref{thm: stacking2} can fail for Liouville cobordisms just for topological reasons. If $k$-handles are present for $k \ge n+1$, then the attaching spheres of different $W_i$ may intersect, even after generic smooth perturbation, and  it may be
impossible to construct $W$ even smoothly. We do not know whether Liouville cobordisms can be stacked when this topological obstruction is absent.
\begin{question}\label{question: liouville_stacking}
 Suppose $W_1, W_2$ are smoothly trivial Liouville cobordisms with $\partial_- W_1 = \partial_- W_2 = 
 (Y, \xi)$. Do there exist (smoothly trivial) Liouville cobordisms $C_1, C_2$ with $\partial_- C_1 = \partial_+ W_1, \partial_- C_2 = \partial_+ W_2$ such that 
 $W_1 \circ C_1$ is Liouville homotopic to $W_2 \circ C_2$?
\end{question}
It seems likely that the proof of Theorem \ref{thm: stacking2} carries over to slightly more general structures than Weinstein structures. The proof requires the Liouville skeleton to be half-dimensional and hence probably holds whenever this is satisfied, e.g. Morse-Bott Weinstein
structures and Liouville structures whose skeleton is stratified by isotropics. 

Finally, we note that the stacking and concatenation operations might not coincide even when both are defined. 
Consider two Weinstein cobordisms $W_1, W_2$ such that 
$\partial_-W_1 = \partial_- W_2 = (Y, \xi)$ and
$\partial_+W_1 = \partial_-W_2$ so that it is possible to form both $Stack_{(Y,\xi)}(W_1,W_2)$ and $W_1 \circ W_2$. Then
$W_1 \circ W_2$ is a well-defined Weinstein cobordism depending just on $W_1, W_2$ while $Stack_{(Y,\xi)}(W_1,W_2)$ is not well-defined
and hence in general they will not agree. We do not know when it is possible to define
$Stack_{(Y,\xi)}(W_1,W_2)$ so that it agrees with $W_1 \circ W_2$. As we explain in Proposition \ref{prop: all_domains_realized} if
$W_2$ is smoothly trivial, then $W_1 \circ W_2$ is Weinstein homotopic to $Stack_{(Y,\xi)}(W_1,W_3)$ for some other Weinstein cobordism $W_3$ which is possibly different from $W_2$. The issue is that when we try to push the Weinstein cobordism $W_2$ in $W_1 \circ W_2$ down to the common contact manifold $\partial_-W_1$, we need to make some choices involving handle-slides and it is not clear that the resulting cobordism will be Weinstein homotopic to $W_2$ (as opposed to some other $W_3$).

\section{Stacking Weinstein domains}\label{sec: maximal_weinstein_domains} 
\subsection{Maximal Weinstein domains}
Our first application of stacking is to construct maximal Weinstein domains 
containing many Weinstein domains as subdomains. In previous work \cite{Lazarev_critical_points}, the author showed that the flexible  domain $W_{flex}$ is a minimal element of $\mathfrak{Weinstein}(W,J)$. 
\begin{theorem}\cite{Lazarev_critical_points}\label{thm: prev_crit_points}
Any Weinstein domain $W^{2n}, n \ge 3,$ can be Weinstein homotoped to 
$W_{flex}^{2n} \circ C^{2n}$ for some smoothly trivial Weinstein cobordism $C^{2n}$.
\end{theorem}
Hence the smooth topology of $W^{2n}$ can be transferred to a symplectically trivial domain $W_{flex}^{2n}$ while the symplectic topology of $W^{2n}$ can be transferred to the smoothly trivial cobordism $C^{2n}$. 
These smoothly trivial cobordisms be stacked without changing the smooth topology of $W^{2n}$. Using Theorem \ref{thm: prev_crit_points} and the stacking construction, we prove the following result, a slightly stronger version of Theorem \ref{thm: finitely_many_subdomains} from the Introduction. 
\begin{theorem}\label{thm: finitely_many_subdomains2} 
For any almost symplectomorphic Weinstein domains $W_1^{2n}, \cdots, W_k^{2n}, n \ge 3,$ there exists an almost symplectomorphic Weinstein domain $W^{2n}$ 
	such that all $W_i^{2n}$ are Weinstein subdomains 
	of $W^{2n}$
	and $W^{2n}\backslash W_i^{2n}$ is a smoothly trivial Weinstein cobordism. Furthermore, the subdomains $W_i^{2n}$'s cover $W^{2n}$.
\end{theorem}
\begin{proof}
By Theorem \ref{thm: prev_crit_points}, we can write $W_i$ as $W_{i, flex} \circ X_i$, where $X_i$ is a smoothly trivial Weinstein cobordism. Since $W_i$ are almost symplectomorphic, so are $W_{i,flex}$. In addition,  $W_{i,flex}$ are flexible and so by the uniqueness h-principle \cite{CE12}  they are Weinstein homotopic;  we will identity them with a fixed $W_{flex}$ and view  $W_{flex}$ as a common subdomain of all the $W_i$. 
Then we set $W: = Stack_{W_{flex}}(W_1, \cdots, W_k)$; equivalently, 
 $W = W_{flex} \circ X$, where $X:=Stack_{\partial W_{flex}}(X_1, \cdots, X_k)$. 
In fact, the stacking construction shows that 
$X = X_i \circ C_i$ for some smoothly trivial Weinstein cobordisms $C_i$ and so $W$ is Weinstein homotopic to $W_{flex} \circ X_i \circ C_i = W_i \circ C_i$. As a result,  $W_i$ is a Weinstein subdomain of $W$ and $W\backslash W_i = C_i$ is smoothly trivial.
Since any Weinstein structure on a smoothly trivial cobordism is 
almost symplectomorphic to the trivial one, $W$ is almost symplectomorphic to $W_i$ for all $i$.
The fact that the $W_i$'s cover $W$ is a general feature of the stacking construction.
\end{proof}
\begin{remark}
	The claim in Theorem \ref{thm: finitely_many_subdomains2} implicitly involves Weinstein homotopies
	 (as does the definition of Weinstein subdomains). 
	 More precisely, for any Weinstein structures $(W_1, \lambda_1, \phi_1), \cdots, (W_k, \lambda_k, \phi_k)$, there are Weinstein  structures $(W, \lambda_1', \phi_1'), \cdots, (W, \lambda_k', \phi_k')$ 
that are all Weinstein homotopic and $(W_i, \lambda_i, \phi_i)$ is a sublevel set of $(W, \lambda_i', \phi_i')$ 
\end{remark}

Theorem \ref{thm: finitely_many_subdomains2} also holds for almost symplectomorphic Weinstein cobordisms with the same negative contact boundary $(Y^{2n-1}, \xi)$; again, we can take $W: = Stack_{W_{flex}}(W_1, \cdots, W_k)$. Here it is important that we glue along $W_{flex}$. If we glue along $(Y^{2n-1}, \xi)$ and form $ Stack_{(Y, \xi) }(W_1, \cdots, W_k)$, which is possible even without using Theorem \ref{thm: prev_crit_points}, the resulting domain will not be almost symplectomorphic to the original $W_i^{2n}$ (unless $W_i^{2n}$ are smoothly trivial). 
Although we do not have any counterexamples, it also seems unlikely that Theorem \ref{thm: finitely_many_subdomains2} holds for $n = 2$. This dimension differs from high dimensions in that there are 4-dimensional almost Weinstein domains that either have no Weinstein structures or have finitely many Weinstein structures; for $n > 2$, any almost Weinstein domain admits infinitely many different Weinstein structures \cite{CE12, MM}.
 
The domain $W^{2n}$ constructed in Theorem \ref{thm: finitely_many_subdomains2} depends on $W_1^{2n}, \cdots, W_k^{2n}$ (as well as other choices) and we do not know if there is a single  Weinstein domain containing \textit{all} almost symplectomorphic Weinstein domains as subdomains. 
However the next result, Corollary \ref{cor: infinite_type} from the Introduction, 
shows that there is a Weinstein \textit{manifold} containing all Weinstein domains as subdomains.
\begin{corollary}\label{cor: infinite_type2}
	For any almost Weinstein domain $(W^{2n},J), n\ge 3$, there exists a maximal Weinstein manifold $W_{max}^{2n}$ almost symplectomorphic to the completion of $(W, J)$ such that any Weinstein domain almost symplectomorphic to $(W, J)$ 
	is a  subdomain of $W^{2n}_{max}$. 
\end{corollary}
\begin{proof}
There are countably many Weinstein \textit{domain} structures  on any almost Weinstein domain. This is because each Weinstein domain has finitely many Weinstein handles and the handles are attached along isotropic spheres which admit finite-dimensional approximations. We can enumerate these Weinstein domains as $\{W_i\}_{i \ge 1}$. 
Then $W_{max}^{2n} = Stack_{W_{flex}}(\{W_i\}_{i\ge 1})$
is a Weinstein manifold
and, as before, it is almost symplectomorphic to (the completion of) $W_i$ and contains all $W_i$ as subdomains by construction.
\end{proof}
On the other hand,  there are open manifolds with uncountably many Weinstein manifold structures \cite{Abouzaid_Seidel}. Consequently we do not know if there is a maximal Weinstein manifold containing all other almost symplectomorphic Weinstein manifolds as submanifolds.

 Theorem \ref{thm: finitely_many_subdomains2} can be reformulated 
as follows: any two almost symplectomorphic Weinstein domains becomes symplectomorphic after attaching a smoothly trivial Weinstein cobordism to each. In fact, attaching a single $n$-Weinstein handle suffices.
  \begin{corollary}\label{cor: single_handle}
If $W_1^{2n},\cdots, W_k^{2n}, n \ge 3,$ are almost symplectomorphic Weinstein domains, then there are Legendrian spheres $\Lambda_i \subset \partial W_i^{2n}$ such that $W_i^{2n} \cup H_{\Lambda_i}^n$ are Weinstein homotopic.
 \end{corollary}
 \begin{remark}
 Of course $\Lambda_i$ depend on the $k$-tuple $W_1, \cdots, W_k$. For example, it seems impossible to fix $\Lambda_1$ a priori independent of $W_2, \cdots, W_k$.
 \end{remark}
 \begin{proof}[Proof of Corollary \ref{cor: single_handle}]
 By Theorem \ref{thm: finitely_many_subdomains2}, there are smoothly trivial Weinstein cobordisms $C_i$ such that $W_i \circ C_i$ are all Weinstein homotopic to a fixed Weinstein domain $W$. Let $\Lambda \subset \partial W$ be any Legendrian sphere and let $C$ be the Weinstein cobordism with $\partial_- C = \partial W$ given by attaching a Weinstein handle along $\Lambda$ to $W$.
 Then $W_i \circ C_i \circ C$ are all Weinstein homotopic to $W \circ C$. 
 The cobordism $C_i \circ C$ admits a smooth Morse function with a single critical point (of index $n$). Hence by 
 a stronger version of Theorem 
\ref{thm: prev_crit_points}  with control over the Weinstein Morse function  \cite{Lazarev_critical_points}, $C_i \circ C_0$ also admits a Weinstein presentation with a single handle (also of index $n$) and therefore is given by attaching a Weinstein handle to some Legendrian $\Lambda_i \subset \partial W_i$.
  Then $W_i \cup H^n_{\Lambda_i} = W_i \circ C_i \circ C_0$ are all Weinstein homotopic to $W  \circ C$ as desired.
  \end{proof}

Theorem \ref{thm: finitely_many_subdomains2} can be generalized to the case when $W_i$ have different topology. In this case, there is no natural choice for the smooth topology of the domain $W$ containing all $W_i$. Hence we will pick an arbitrary an almost Weinstein domain $X$ and construct a Weinstein domain $W$ almost symplectomorphic to $X$ that contains $W_i$. Of course, it is necessary that $W_i$ admit \textit{almost} Weinstein embeddings into $X$, i.e.  
$X\backslash W_i$ is an almost Weinstein cobordism. The following result shows that this necessary condition is sufficient. 
 \begin{corollary}\label{cor: subdomains_diff_top}
 If $W_1^{2n}, \cdots, W_k^{2n}, n \ge 3,$ are Weinstein domains that have almost Weinstein embeddings into an almost Weinstein domain $X^{2n}$, then there exists a Weinstein domain almost symplectomorphic to $X^{2n}$ containing $W_i^{2n}$ as Weinstein subdomains. 
 \end{corollary}
 \begin{remark}
 	The  domains $W_i$ do not need to have the same topology, unlike in Theorem \ref{thm: finitely_many_subdomains2}.  This result generalizes Theorem 6.4 of \cite{eliashberg_revisited}, which considers the case when $k = 1$ and $W_1^6$ is some Weinstein domain with a single index 3 critical point and $X = T^*S^3$. 
 \end{remark}
 \begin{proof}
 By Eliashberg's existence h-principle \cite{CE12} for Weinstein domains, $X\backslash W_i^{2n}$ has a flexible Weinstein structure which we denote $(X\backslash W_i^{2n})_{flex}$. Let $X_i := W_i^{2n} \circ (X\backslash W_i^{2n})_{flex}$. Then $X_i$ are almost symplectomorphic Weinstein structures on $X$ containing $W_i$ as Weinstein-subdomains.  By Theorem \ref{thm: finitely_many_subdomains2}, there exists a Weinstein structure on $X$ that contains all $X_i$ as Weinstein subdomains and hence $W_i$ as Weinstein subdomains. 
 \end{proof}
 
In the previous results Theorem \ref{thm: finitely_many_subdomains2} and Corollary \ref{cor: subdomains_diff_top}, we started with abstract Weinstein domains and constructed a Weinstein domain containing them as subdomains. Now we start with Weinstein subdomains of some Weinstein domain, perhaps produced via these previous results. We will assume that the subdomains are abstractly symplectomorphic but not symplectomorphic via a symplectomorphism of the ambient domain. The following result shows that in some sense these subdomains become symplectomorphic in a larger domain. 
 \begin{corollary}\label{cor: abstractly_symp}
Let $\phi_1, \cdots, \phi_k: W_0^{2n} \hookrightarrow W^{2n}, n \ge 3,$ be formally isotopic Weinstein embeddings. Then there is a Weinstein domain $X^{2n}$ almost symplectomorphic to $W^{2n}$ 
and formally isotopic Weinstein embeddings $\psi_1, \cdots, \psi_k: W^{2n} \hookrightarrow X^{2n}$ such that $\psi_1 \circ \phi_1 = \cdots = \psi_k \circ\phi_k$.  
 \end{corollary}
 \begin{proof}
 Consider the almost symplectomorphic Weinstein cobordisms $C_i: = W^{2n} \backslash \phi_i(W_0^{2n})$. Using the symplectomorphism between $W_0$ and $\phi_i(W_0)$, we identify $\partial_- C_i = \partial(\phi_i(W_0))$
 with $\partial W_0$. Then we can construct the Weinstein cobordism $C := Stack_{C_{flex}}(C_1, \cdots, C_k)$ as in Theorem \ref{thm: finitely_many_subdomains2} and set $X^{2n}: = W_0^{2n} \circ C$; equivalently, $X^{2n} = Stack_{W_0 \circ C_{flex}}((W, \phi_1'), \cdots, (W, \phi_k'))$, where $\phi_i': W_0 \circ C_{flex} \hookrightarrow W$ is an extension of $\phi_i: W_0 \hookrightarrow W$.  Since $C$ is almost symplectomorphic to $C_i$, the domain $X$ is almost symplectomorphic to $W$. 
There are Weinstein embeddings $\psi_i: W \hookrightarrow X$ obtained by identifying the Weinstein homotopic domains $W$ and $W_0 \circ C_i$ and taking the embedding $W_0 \circ C_i \hookrightarrow W_0 \circ C = X$ induced by the inclusion $C_i \subset C = Stack_{C_{flex}}(C_1, \cdots, C_k)$. The embeddings $\psi_i$ are all formally isotopic since $W_i \hookrightarrow W$ are all formally isotopic. Finally,  $\psi_i \circ \phi_i: W_0 \hookrightarrow X$ is independent of $i$ since it coincides with the inclusion $W_0 \subset W_0 \circ C_i \subset W_0 \circ C$.  
 \end{proof}
\begin{remark} 
As in Theorem \ref{thm: finitely_many_subdomains2}, there are implicit Weinstein homotopies in  Corollary \ref{cor: abstractly_symp}. 
The statement that $\phi_i(W_0^{2n}) \subset W^{2n}$ are Weinstein subdomains involves Weinstein homotopies of $W^{2n}$. In addition, to define $\psi_i$, we identified the Weinstein homotopic domains $W_0 \circ C_i$ and $W$. So the maps $\psi_i$ have sources $W_0 \circ C_i$ that are \textit{Weinstein homotopic} to $W$ but are not actually $W$ itself. Therefore a more precise version of  Corollary \ref{cor: abstractly_symp} is that there are Weinstein domains $W_i= W_0 \circ C_i$ and subdomains $\phi_i': W_0 \hookrightarrow W_i$ such that the pair $(W_i, (W_0, \phi_i')) $ is Weinstein homotopic to $(W, (W_0, \phi_i))$ and  Weinstein embeddings $\psi_i': W_i \hookrightarrow X^{2n}$ such that the maps $\psi_i' \circ \phi_i' : W_0 \hookrightarrow X^{2n}$ all agree. 
\end{remark}

Since we stack the cobordisms $W \backslash \phi_i(W_0)$, 
Corollary \ref{cor: abstractly_symp} can be thought of as a relative version of Theorem \ref{thm: finitely_many_subdomains2} for Weinstein cobordisms.
In general the embeddings $\phi_i: W \hookrightarrow X$ will not be Hamiltonian isotopic for different $i$. Otherwise, there would be a single embedding $\psi: W \hookrightarrow X^{2n}$ such that $\psi\circ \phi_i: W_0 \hookrightarrow X^{2n}$ are Hamiltonian isotopic in $X^{2n}$ even though $\phi_i: W_0 \hookrightarrow W$ are not necessarily Hamiltonian isotopic in $W$. This is impossible in general; see Section \ref{ssec: domains_many_lag} for examples.
We also note that Corollary \ref{cor: abstractly_symp} can be generalized to allow Weinstein embeddings $\phi_i: W_0 \hookrightarrow W_i$, where $W_i$ are different  Weinstein domains; in this case, we would get Weinstein embeddings $\psi_i: W_i^{2n} \hookrightarrow X^{2n}$ such that $\psi_i \circ \phi_i$ all agree.

\subsection{Stacking operation}\label{ssec: stacking_operation}
 
Now we will present some more results on the stacking operation and elaborate on Theorem \ref{thm: finitely_many_subdomains2}. This operation takes in almost symplectomorphic Weinstein domains and produces an almost symplectomorphic Weinstein domain containing them as subdomains. For trivial reasons, any Weinstein domain can be obtained by stacking almost symplectomorphic Weinstein domains.
For example, if we view $W_{flex}$ as a trivial subdomain of $W_{flex} = W_{flex} \circ (Y, \xi) \times [0,1]$, then $Stack_{W_{flex}}(W_{flex} \circ (Y, \xi) \times [0,1], W)  = W$; even more trivially, we have $Stack_W(W, W) = W$. However stacking is not a well-defined operation; indeed, for any Weinstein domain, there are many Weinstein domains that contains it as a subdomain.
Precisely because  stacking is not well-defined, it is possible to produce interesting Weinstein domains 
by stacking just flexible domains. The following result 
 is a sort of converse to Theorem \ref{thm: finitely_many_subdomains2}. 
\begin{proposition}\label{prop: all_domains_realized}
Let $W_0^{2n}, n \ge 3,$ be a Weinstein subdomain of $W^{2n}$ such that $W^{2n} \backslash W_0^{2n}$ is smoothly trivial
and $\pi_1(W) = \pi_1(W_0) = 0$. Then $(W, W_0)$ is Weinstein homotopic to $(Stack_{W_{flex}}(W_0, W_{flex}), W_0)$. 
\end{proposition}
\begin{remark}\label{rem: all_domains_realized}
We do not know whether the condition $\pi_1(W) =\pi_1(W_0)= 0$ is essential; it may be possible to drop this condition by
more carefully controlling the smooth isotopy class of the Legendrian attaching spheres. 
We also note that to construct the domain
$Stack_{W_{flex}}(W_0, W_{flex})$, we need an embedding $i: W_{flex} \hookrightarrow W_{flex}$. 
As we will see in the proof, the cobordism $W_{flex} \backslash i(W_{flex})$
is Weinstein homotopic to the trivial Weinstein cobordism without any critical points but does not actually coincide with this cobordism; namely  $W_{flex} \backslash i(W_{flex})$ has some symplectically cancelling handles. 
\end{remark}

\begin{proof}
As in the proof of Theorem \ref{thm: finitely_many_subdomains2}, we first present $W^{2n}_0$ as $W^{2n}_{flex} \circ C_0^{2n}$ for some smoothly trivial Weinstein cobordism $C_0^{2n}$. Also, let $C^{2n} := W^{2n} \backslash W_0^{2n}$, which is a smoothly trivial by assumption; so we have 
$W^{2n} = W_0^{2n} \circ C^{2n} = W_{flex}^{2n} \circ C_0^{2n} \circ C^{2n}$. The stronger version of Theorem \ref{thm: prev_crit_points} in \cite{Lazarev_critical_points} shows that $C_0^{2n}, C^{2n}$ both have Weinstein presentations with two handles. Namely, 
$C^{2n}_0 = H^{n-1}_0 \cup H^n_{\Lambda_0}, C^{2n} = H^{n-1} \cup H^n_{\Lambda}$,  and that the attaching spheres of $H^{n-1}_0, H^{n-1}$ are contained in a Darboux ball of $\partial W_{flex}^{2n}$; if $W_0^{2n} = W_{flex}^{2n}$, we take  $H^{n-1}_0, H^n_{\Lambda_0}$ to be symplectically cancelling. We view $\Lambda_0, \Lambda$ as Legendrians in $\partial (W_{flex}^{2n} \cup H^{n-1}_0 \cup H^{n-1})$. 

Since $C_0$ is smoothly trivial, $\Lambda_0$ has algebraic intersection one with the belt sphere of $H^{n-1}_0$. Furthermore, $\Lambda_0$ is disjoint from the belt sphere of $H^{n-1}$  since $H^{n-1}$ is part of the cobordism $W\backslash W_0$, which is attached after $H^{n}_{\Lambda_0}$. 
Similarly since $C^{2n}$ is also smoothly trivial, $\Lambda$ has algebraic intersection one with $H^{n-1}$. However, $\Lambda$ may have non-trivial intersection with $H^{n-1}_0$ since $H^n_{\Lambda}$ is attached after $H^{n-1}_0$. By  handle-sliding $\Lambda$ over $\Lambda_0$ possibly several times, we get a new Legendrian $\Lambda'$ such that the algebraic intersection  of $\Lambda'$ and the belt sphere of $H^{n-1}_0$ is zero (and the algebraic intersection of $\Lambda'$ with the belt sphere of $H^{n-1}$ is still one). We can also perform these handleslides so that $\Lambda' \subset \partial (W_{flex} \cup H^{n-1}_0\cup H^{n-1})$ is loose in the complement of the belt sphere of $H^{n-1}$; see \cite{Lazarev_critical_points} for the relationship between handle-slides and looseness. 
Since $\Lambda'$ has algebraic intersection number zero with the belt sphere of $H^{n-1}_0$ and everything is simply-connected, we can use the Whitney trick to obtain a smooth isotopy that displaces $\Lambda'$ from this sphere. Furthermore, we can use Whitney disks that are disjoint from the belt sphere of $H^{n-1}$ and hence assume that this smooth isotopy is supported away from the belt sphere of $H^{n-1}$. Since $\Lambda'$ is loose in the complement of this belt sphere (but not in the complement of $\Lambda_0$), there is a Legendrian isotopy of $\Lambda'$ that displaces $\Lambda'$ from the belt sphere of $H^{n-1}_0$ and is supported away from the belt sphere of $H^{n-1}$. We can extend this Legendrian isotopy to an ambient contact isotopy $\phi_t$ such that $\phi_t$ is supported away from the belt sphere of $H^{n-1}$; let $\phi: = \phi_1$. In particular, $\phi(\Lambda')$ is disjoint from the belt sphere of $H^{n-1}_0$ by construction and is loose in 
$\partial(W_{flex} \cup H^{n-1})$.
Furthermore, $\phi(\Lambda_0)$ is disjoint from the belt sphere of $H^{n-1}$ since $\Lambda_0$ was disjoint from this sphere and the contact isotopy $\phi_t$ is supported away from this sphere. 

Now we consider the domain 
$W_{flex} \cup H^{n-1}_0 \cup H^{n-1}_1 \cup H^n_{\phi(\Lambda_0)} \cup H^n_{\phi(\Lambda')}$. This domain is Weinstein homotopic to $W$ since it is obtained from the original presentation of $W$ by handle-slides and Legendrian isotopies of the attaching spheres. 
Furthermore, since $\Lambda_0, \Lambda'$ are disjoint from the belt spheres of $H^{n-1}_0, H^{n-1}$ respectively, this domain is $Stack_{W_{flex} } (W_{flex} \cup H^{n-1}_0 \cup H^n_{\phi(\Lambda_0)}, W_{flex} \cup H^{n-1} \cup H^n_{\phi(\Lambda') })$.  The first domain 
$W_{flex} \cup H^{n-1}_0 \cup H^n_{\phi(\Lambda_0)}$
is Weinstein homotopic to $W_0$ since the new attaching sphere $\phi(\Lambda_0)$ and the old attaching sphere $\Lambda_0$ differ just by Legendrian isotopy. Since the handle-slides are done on top of $W_0$, the inclusion 
$W_0 = W_{flex} \cup H^{n-1}_0  \cup H^n_{\phi(\Lambda_0)} \subset W = W_{flex} \cup H^{n-1}_0 \cup H^n_{\phi(\Lambda_0)} 
\cup H^{n-1} \cup H^n_{\phi(\Lambda')}$ coincides with the original embedding $W_0 \subset W$, i.e. 
the pair 
$(Stack_{W_{flex} } (W_{flex} \cup H^{n-1}_0 \cup H^n_{\phi(\Lambda_0)}, W_{flex} \cup H^{n-1} \cup H^n_{\phi(\Lambda') }), W_{flex} \cup H^{n-1}_0  \cup H^n_{\phi(\Lambda_0)})$ 
is Weinstein homotopic to 
$(W, W_0)$.
Also,
$W_{flex} \cup H^{n-1} \cup H^n_{\phi(\Lambda')}$ is almost symplectomorphic to $W_{flex}$ and $\phi(\Lambda')$ is loose; the former claim follows from the facts that $\phi(\Lambda')$ has algebraic intersection number one with the belt sphere of $H^{n-1}$ and $\pi_1(W) = 0$.
Therefore $W_{flex} \cup H^{n-1} \cup H^n_{\phi(\Lambda')}$ is Weinstein homotopic to $W_{flex}$. 
So 
$(Stack_{W_{flex} } (W_{flex} \cup H^{n-1}_0 \cup H^n_{\phi(\Lambda_0)}, W_{flex} \cup H^{n-1} \cup H^n_{\phi(\Lambda') }), W_{flex} \cup H^{n-1}_0  \cup H^n_{\phi(\Lambda_0)})$ is homotopic to both
$(Stack_{W_{flex}}(W_0, W_{flex}), W_0)$
 and $(W,W_0)$, which proves the claim.
\end{proof}

By taking $W_0^{2n}$ to be $W_{flex}^{2n}$ in Proposition \ref{prop: all_domains_realized}, we get the following result, Corollary \ref{cor: covering} from the Introduction. \begin{corollary}\label{cor: domain_flex_everything}
Any Weinstein domain $W^{2n}, n \ge 3,$ with $\pi_1(W) = 0$ is Weinstein homotopic to
$Stack_{W_{flex}}(W_{flex}^{2n}, W_{flex}^{2n})$.
In particular, 
there exist Weinstein subdomains $\phi_1, \phi_2: W_{flex}^{2n} \hookrightarrow W^{2n}$ such that 
 $\phi_1(W_{flex}^{2n}) \cup \phi_2(W_{flex}^{2n}) = W^{2n}$. 
\end{corollary}
This result is possible precisely because the stacking operation depends on more than just the Weinstein homotopy type of the stacked domains. The extra data of $W$ is related to  the extra data needed to define 
$Stack_{W_{flex}}(W_{flex}^{2n}, W_{flex}^{2n}) = 
Stack_{W_{flex}}(
(W_{flex}^{2n}, i_1), (W_{flex}^{2n}, i_2))$, namely the linking of the attaching spheres of $W_{flex} \backslash i_1(W_{flex})$ and $W_{flex} \backslash i_2(W_{flex})$. This is why 
$W_{flex} \backslash i(W_{flex})$ is not the trivial Weinstein cobordism, as noted in Remark \ref{rem: all_domains_realized}.

In some sense, the flexible domain $W_{flex}^{2n}$ 
is the only domain that can appear  in the statement of Corollary \ref{cor: domain_flex_everything}.  If all Weinstein domains were homotopic to $Stack_{W_{flex}}(W_{flex}, W_0)$ or $Stack_{W_0}(W_0, W_0)$ for some fixed $W_0$,  then $W_0$ would have to be a subdomain of all Weinstein domains and hence have to be (sub)flexible.
However if $W_0$ is already a Weinstein subdomain of $W$, then such a result is indeed possible and a variation on Corollary \ref{cor: domain_flex_everything} 
shows that 
$(W, W_0)$ is Weinstein homotopic to
$(Stack_{W_0}(W_0, W_0), W_0)$. So if $W$ contains an almost symplectomorphic Weinstein subdomain $W_0$, then $W$ can be covered by two Weinstein embeddings of $W_0$ into $W$.

We do not know whether an analogous version of Proposition \ref{prop: all_domains_realized} holds for \textit{multiple} subdomains. That is, Theorem \ref{thm: finitely_many_subdomains2} shows that any finite collection of domains can be stacked to produce a domain containing them as subdomains but it is unclear whether the converse holds.
As we explained before, any domain obtained by stacking subdomains can be covered by those subdomains. Of course, two arbitrary subdomains will not cover the ambient Weinstein domain and hence we should not require the stacked domain to cover all of the ambient domain. 
\begin{question}\label{question: two_subdomains}
Suppose that $W$ contains subdomains $W_1, W_2$. Then is 
there a subdomain $V \subset W_i$ such that  $Stack_{V}(W_1, W_2) \subset W$ and  $W_i \subset 
Stack_{V}(W_1, W_2) \subset W$ coincide with the original Weinstein embeddings $W_i \subset W$?
\end{question}
The situation in Theorem \ref{thm: finitely_many_subdomains2} where we produce domains containing many subdomains is quite special. The subdomains are all sublevel sets for Weinstein Morse functions with the \textit{same} Liouville vector field. However, it is not clear that arbitrary subdomains of a given Weinstein domain should satisfy this property. 
Indeed,  for any two Weinstein subdomains $W_1, W_2 \subset W$ with this property, the answer to Question \ref{question: two_subdomains} is yes. If $W_i$ are sublevel sets for Weinstein Morse functions $\phi_i$ of $W$ with the same Liouville vector field $X$, then 
$\phi_1 + \phi_2$ and a smoothing of $\max\{\phi_1, \phi_2\}$ are also such Weinstein Morse functions. Then $\phi_1 + \phi_2$ has a sublevel set $V$ that is a common subdomain of $W_1, W_2$ and $\max\{\phi_1, \phi_2\}$ has a sublevel set $Stack_V(W_1, W_2)$, which is therefore a subdomain of $W$. 

If we have \textit{nested} subdomains $W_1 \subset W_2 \subset \cdots \subset W_k$, then clearly there exists a single Weinstein function $\phi$ (with a single Liouville vector field) and regular values $c_1 < \cdots < c_k$  of $\phi$ such that 
$\{\phi \le c_i\} = W_i$.
In this case, all these subdomains are obtained by stacking. 
\begin{corollary}
If $W_1^{2n} \subset W_2^{2n} \subset \cdots \subset W_k^{2n}, n \ge 3,$ are nested almost symplectomorphic domains such that $W_{i+1} \backslash W_i$ are smoothly trivial and $\pi_1(W_i) = 0$, then $W_i$ is Weinstein homotopic to 
$Stack_{W_{flex}}(W_1, W_{flex, 1}, \cdots, W_{flex,i-1})$, the stacking $W_1$ with $i-1$ copies of $W_{flex}$. Furthermore, $W_{i-1} \subset W_i$ corresponds to the natural inclusion of stacked domains. 
\end{corollary}
\begin{proof}
By Proposition \ref{prop: all_domains_realized}, $W_{i+1} = Stack_{W_{flex}}(W_i, W_{flex})$ and so the result follows by induction and the fact that 
$Stack_{(Y, \xi)}(X_1, \cdots, X_i, Stack_{(Y, \xi)}(X_{i+1}, \cdots, X_{k}))$ is Weinstein homotopic to 
$Stack_{(Y, \xi)}(X_1, \cdots, X_i, X_{i+1}, \cdots, X_k)$ for any Weinstein cobordisms $X_i$.\end{proof}

\section{Stacking contact manifolds}\label{sec: maximal_contact}

\subsection{Maximal contact manifolds}\label{ssec: max_contact}
We now prove some contact analogs of the results in Section \ref{sec: maximal_weinstein_domains} for Weinstein domains.
To prove the existence of maximal Weinstein domains, we used Theorem \ref{thm: prev_crit_points}: any Weinstein domain can be decomposed into a fixed Weinstein domain (its flexiblization)  depending just on the smooth topology plus a smoothly trivial Weinstein cobordism. On the contact side, Casals, Murphy, and Presas \cite{CMP} proved that any contact manifold of dimension at least five is the positive end of a smoothly trivial Weinstein cobordism from a fixed (overtwisted) contact manifold. Using their result, we will prove that any finite collection of contact manifolds has a maximal element. We first assume that the contact manifolds are almost contactomorphic and prove the second claim in Theorem \ref{thm: contact_cobordisms}
from the Introduction.
We later remove this assumption in Corollary \ref{cor: contact_maximal_weinstein} below. 
 \begin{theorem}\label{thm: contact_cobordisms2}
If $(Y_1^{2n-1}, \xi_1), \cdots, (Y_k^{2n-1}, \xi_k), n \ge 3,$ are almost contactomorphic contact manifolds, then there exist smoothly trivial Weinstein cobordisms $C_i^{2n}$ such that $\partial_-C_i^{2n} = (Y_i^{2n-1}, \xi_i)$ and $\partial_+C_i^{2n}$ are all contactomorphic.
\end{theorem}
\begin{proof}
Let $(Y, \xi_{ot,i})$ be an overtwisted contact structure in the 
almost contact class defined by $(Y_i, \xi_i)$. 
Casals, Murphy, and Presas \cite{CMP} showed that there is a smoothly trivial Weinstein cobordism $X_i$ from $(Y, \xi_{ot,i})$ to $(Y_i, \xi_i)$. 
Almost contactomorphic overtwisted contact structures are actually contactomorphic \cite{BEMtwisted} and so we can identify $(Y, \xi_{ot,i})$ with a fixed contact structure 
$(Y, \xi_{ot})$.
Since $\partial_-X_i = (Y, \xi_{ot})$ all agree, we can use Theorem \ref{thm: stacking2} to form the Weinstein cobordism $X:=Stack_{(Y, \xi_{ot})}(X_1, \cdots, X_k)$ with $\partial_-X = (Y, \xi_{ot})$. Furthermore, 
Theorem \ref{thm: stacking2} provides Weinstein cobordisms $C_i$ with $\partial_- C_i = \partial_+ X_i$ such that  $X_i \circ C_i$ is Weinstein homotopic to $X$. In particular, $\partial_+C_i = \partial_+X$ for all $i$.
Since the $X_i$'s are smoothly trivial, by Theorem \ref{thm: stacking2} so are the $C_i$'s.
\end{proof}

Theorems \ref{thm: contact_cobordisms2} fails for $n =2$ and so the restriction $n \ge 3$ is necessary. For example, Seiberg-Witten theory can be used to show that there is no smoothly trivial 4-dimensional Weinstein cobordism from an overtwisted contact structure to a fillable contact structure  \cite{Mrowka_concordance_overtwisted}. If $(Y_1, \xi_1)$ is overtwisted and $(Y_2, \xi_2)$ is fillable, then a 4-dimensional version of Theorem \ref{thm: contact_cobordisms} would provide smoothly trivial Weinstein cobordisms from $(Y_1, \xi_1), (Y_2, \xi_2)$ to $(Y, \xi)$, which would be fillable via the filling of  $(Y_2, \xi_2)$. However a weaker version of Theorem \ref{thm: contact_cobordisms2} does hold in this dimension. 
\begin{theorem}\label{thm: contact_cob_dimension4}
If $(Y_1^{3}, \xi_1), \cdots, (Y_k^{3}, \xi_k)$ are almost contactomorphic contact manifolds, then there are Weinstein cobordisms $C_i^4$
with $k-1$ Weinstein 2-handles and no Weinstein $0,1$-handles
such that $\partial_-C_i^4 = (Y_i, \xi_i)$ and  $\partial_+C_i^4$ are all contactomorphic. 
\end{theorem}
\begin{proof}
We will  use a modified version of the Casals-Murphy-Presas result that holds in dimension 3. Namely, if $\Lambda_i \subset (Y_i^3, \xi_i)$ is a once stabilized Legendrian unknot contained in a Darboux chart, then $(Y_i^3, \xi_i) \cup H^2_{\Lambda_i}[+1]$, the result of doing $+1$ contact surgery along $\Lambda_i$, is overtwisted  \cite{Eliashberg_20_years}. 
Furthermore, since 
the $\Lambda_i$ are formally Legendrian isotopic, the 
$(Y_i, \xi_i)\cup H^2_{\Lambda_i}[+1]$ are all almost contactomorphic. Hence by Eliashberg's h-principle \cite{Eliashberg_overtwisted3} for overtwisted contact structures in dimension 3,  $(Y_i, \xi_i)\cup H^2_{\Lambda_i}[+1]$ are all contactomorphic
to a fixed $(Y', \xi_{ot})$. Since 
$(Y_i, \xi_i)\cup H^2_{\Lambda_i}[+1]$ are obtained from $(Y_i, \xi_i)$ via anti-surgery, 
there are Weinstein cobordisms $X_i^4$ with a single Weinstein 2-handle 
such that $\partial_- X_i^4 = (Y', \xi_{ot})$ and $\partial_+X_i^4 = (Y^3, \xi_i)$.
Now we can proceed as in the proof of Theorem \ref{thm: contact_cobordisms2} and form the cobordisms $X^4$ and $C_i^4$. Note that $C_i^4 := Stack_{(Y', \xi_{ot})}(X_1^4, \cdots, X_k^4) \backslash X_i^4$ has $k-1$ Weinstein 2-handles and no $0,1$-handles.
\end{proof}

\begin{remark}
Theorem \ref{thm: contact_cob_dimension4} is actually a bit stronger than stated. 
The proof shows that we only need
$(Y_i^3, \xi_i)\cup H^2_{\Lambda_i}[+1]$, but not necessarily $(Y_i^3, \xi_i^3)$,  to be almost contactomorphic; as a result, the $Y_i$'s need not be diffeomorphic.
\end{remark}

By using Theorem \ref{thm: contact_cobordisms2} and Theorem \ref{thm: contact_cob_dimension4}, we get the contact analog of 
Corollary \ref{cor: single_handle}. 
\begin{corollary}\label{cor:contact_single_handle}
	If $(Y_1^{2n-1}, \xi_1), \cdots, (Y_k^{2n-1}, \xi_k), n \ge 3,$ are almost contactomorphic, then there exist Legendrian spheres $\Lambda_i^{n-1} \subset (Y_i^{2n-1}, \xi_i)$ such that $(Y_i^{2n-1}, \xi_i) \cup  H^n_{\Lambda_i}$ are contactomorphic. If $n = 2$, then the same holds for $k = 2$. 
\end{corollary}

Now we apply Corollary
\ref{cor:contact_single_handle} to the case when  $(Y, \xi)$ has an almost Weinstein filling  $(W,J)$ and prove Corollary \ref{cor: contact_surgery_presentation} from the Introduction.
\begin{corollary}\label{cor: contact_surgery_presentation2}
If $(Y^{2n-1}, \xi), n \ge 3,$ has an almost Weinstein filling, then $(Y^{2n-1}, \xi)$ can be obtained from $(S^{2n-1}, \xi_{std})$ by a sequence of contact surgeries (of any index at most $n$) and a single contact anti-surgery (of index $n$). 
\end{corollary}
\begin{proof}
Suppose that $(Y^{2n-1}, \xi)$ has an almost Weinstein filling $(W^{2n}, J)$. By Eliashberg's existence h-principle \cite{Eli90} (which holds for $n \ge 3$), there is a flexible Weinstein structure almost symplectomorphic to $(W^{2n}, J)$, which we will also denote by $W^{2n}$. Then $(Y^{2n-1}, \xi)$ and $\partial W^{2n}$ are almost contactomorphic and so by Corollary \ref{cor:contact_single_handle}, there are Legendrians $\Lambda_1 \subset (Y^{2n-1}, \xi), 
\Lambda_2 \subset \partial W^{2n}$ such that 
$(Y^{2n-1}, \xi) \cup H^n_{\Lambda_1} = \partial W \cup H^n_{\Lambda_2}$. Let $\Lambda_1' \subset (Y^{2n-1}, \xi) \cup H^n_{\Lambda_1}$ be the belt sphere of the Weinstein handle $H^n_{\Lambda_1}$ so that 
$(Y^{2n-1}, \xi) = (Y^{2n-1}, \xi) \cup H^n_{\Lambda_1} \cup 
H^n_{\Lambda_1'}[+1]$. We can also view $\Lambda_1'$ as a Legendrian of $\partial W \cup H^n_{\Lambda_2}$ by the previous identification. Then $(Y^{2n-1}, \xi) = \partial W \cup H^n_{\Lambda_2} \cup H^n_{\Lambda_1'}[+1]$. Since $\partial W$ has a Weinstein filling $W^{2n}$, it is obtained  from $(S^{2n-1}, \xi_{std})$ by contact surgery. Therefore 
$\partial W \cup H^n_{\Lambda_2} \cup H^n_{\Lambda_1'}[+1]$ is obtained from 
 $(S^{2n-1}, \xi_{std})$ by contact surgery and a single contact anti-surgery (of index $n$). 
\end{proof}

There is also an analog of Corollary \ref{cor: subdomains_diff_top} for contact manifolds with different topology. 
\begin{corollary}\label{cor: contact_cobordisms_diff_top}
Suppose that $(Y_1^{2n-1}, \xi_1), \cdots, (Y_k^{2n-1}, \xi_k), n \ge 3,$ are contact manifolds, not necessarily with the same topology, and $X_i$ are almost Weinstein cobordisms from $(Y_i, \xi_i)$ to some fixed almost contact manifold $(Y, J)$. Then there are Weinstein cobordisms $Z_i$ almost symplectomorphic to $X_i$ such $\partial_- Z_i = (Y_i, \xi_i)$ and $\partial_+ Z_i$ are all contactomorphic. 
\end{corollary}
\begin{proof}
Let $X_{i,flex}$ be the flexible Weinstein cobordism in the same formal class as $X_i$ with $\partial_-(X_{i,flex}) = (Y_i, \xi_i)$ provided by Eliashberg's existence h-principle.  Then $\partial_+(X_{i,flex})$ are all almost contactomorphic to $(Y, J)$ by assumption. Hence by Theorem \ref{thm: contact_cobordisms}, there exists a smoothly trivial Weinstein cobordism $C_i$ such that $\partial_- C_i = \partial_+(X_{i,flex})$, and $\partial_+C_i$ are contactomorphic for all $i$  (and in almost contact class $(Y,J)$). Then $Z_i := X_{i, flex}\circ C_i$ is a Weinstein cobordism such that  $\partial_- Z_i = \partial_- (X_{i, flex}) = (Y_i, \xi_i)$ and $\partial_+Z_i = \partial_+ C_i$ are all contactomorphic contact structures in $(Y, J)$. Furthermore, $Z_i$ is almost symplectomorphic to $X_i$ since $C_i$ is smoothly trivial. 
\end{proof}

The result of Casals, Murphy, and Presas \cite{CMP}  about smoothly trivial Weinstein cobordisms is actually stronger than stated above. They proved an existence h-principle for Weinstein cobordisms (with arbitrary topology) with overtwisted negative end and prescribed positive end; later this result was generalized by Eliashberg and Murphy \cite{EM}  who proved an h-principle for \textit{Liouville} cobordisms with
overtwisted negative end. On the other hand, Corollary \ref{cor: contact_cobordisms_diff_top} can be viewed as an h-principle for Weinstein cobordisms with prescribed negative end but no control over the positive end. There is no h-principle for Weinstein cobordisms with both negative \textit{and} positive ends prescribed. For example, there is no  Weinstein cobordism from a fillable contact structure to an overtwisted contact structure \cite{BEMtwisted}.

Bowden, Crowley, and Stipsicz
\cite{Bowden_Crowley_Stipsicz_stein_fillable_II} used surgery theory to prove that  almost Weinstein cobordisms $X_i$ satisfying the conditions in Corollary \ref{cor: contact_cobordisms_diff_top} always exist. 
Combining their result with Corollary \ref{cor: contact_cobordisms_diff_top}, we can prove the first claim in Theorem \ref{thm: contact_cobordisms} from the Introduction. \begin{corollary}\label{cor: contact_maximal_weinstein}
For any contact manifolds $(Y_1^{2n-1}, \xi_1), \cdots, (Y_k^{2n-1}, \xi_k), n \ge 3$, there exist Weinstein cobordisms $C_i^{2n}$ such that $\partial_- C_i^{2n} = (Y_i^{2n-1}, \xi_i)$ and $\partial_+C_i^{2n}$ are contactomorphic. 
\end{corollary}
Hence any finite collection of contact manifolds admits a maximal contact element with respect to the Weinstein cobordism relation. Like in the  Weinstein setting, we do not know whether this holds for an arbitrary infinite collection of contact manifolds. 
\begin{question}\label{question: infinite_contact_max}
Which infinite collections of contact manifolds $(Y_i, \xi_i)$ admit a maximal contact structure $(Y_{max}, \xi_{max})$ such that each $(Y_i, \xi_i)$ has a Weinstein cobordism to $(Y_{max}, \xi_{max})$? For example, is this true for $\partial W_{flex, i}^{2n}$, the collection of all flexibly-fillable contact structures in $(Y^{2n-1}, J)$?
\end{question}
Here we allow $(Y_{max}, \xi_{max})$ to depend on the fixed infinite collection of contact structures  $(Y_i, \xi_i)$; hence
$(Y_{max}, \xi_{max})$  may be maximal for this particular collection but not maximal for arbitrary infinite collections, i.e. all contact structures. 
However as we showed in Corollary \ref{cor: maximal_strong_cobordism} in the Introduction, there is a contact manifold that is maximal with respect to the \textit{strong symplectic cobordism} relation for \textit{all} contact manifolds. 

\subsection{Connections to the symplectic mapping class group}\label{ssec: symplectomorphism}

In this section, we discuss a possible connection between the results in Section \ref{ssec: max_contact} and  the symplectic mapping class group. 
We first note that as an immediate application of Corollary \ref{cor: contact_cobordisms_diff_top},  any contact structure is Weinstein cobordant to a Weinstein-fillable contact structure. 
\begin{corollary}\label{cor: cobordant_to_fillable}
For any contact structure $(Y^{2n-1}, \xi), n \ge 3,$ there exists a Weinstein cobordism $C^{2n}$ and Weinstein domain $W^{2n}$ such that 
$\partial_- C^{2n} = (Y^{2n-1}, \xi)$ and $\partial_+ C^{2n}  = \partial W^{2n}$. 
\end{corollary}
\begin{remark}
Here the smooth and symplectic topology of $C^{2n}, W^{2n}$ depend on the smooth and contact topology of $(Y^{2n-1}, \xi)$. 
\end{remark}
\begin{proof}
Take any Weinstein domain $W^{2n}_0$. By Corollary \ref{cor: contact_cobordisms_diff_top}, there exist Weinstein cobordisms $C^{2n}, C_0^{2n}$ such that 
$\partial_- C^{2n} = (Y^{2n-1}, \xi), \partial_- C_0^{2n} = \partial W_0^{2n}$, and $\partial_+C^{2n} = \partial_+C_0^{2n}$. Then 
$W^{2n}:=W_0^{2n} \circ C_2^{2n}$ is a Weinstein domain 
and $\partial W^{2n} = \partial_+ C_2^{2n} = \partial_+C^{2n}$ as desired. 
\end{proof}
The $n =2$ case of Corollary \ref{cor: cobordant_to_fillable} was proven by Etnyre and Honda \cite{Etnyre_Honda_caps}. They used the fact that any 3-dimensional contact manifold has an open book decomposition whose monodromy is given by a product of positive and negative Dehn twists. Then by adding more positive Dehn twists to cancel out the negative Dehn twists, they built a Weinstein cobordism $C^4$ such that $\partial_- C^4 = (Y^3, \xi)$ and $\partial_+ C^4$ has an open book decomposition whose monodromy has only positive Dehn twists and is therefore Weinstein-fillable.
High-dimensional contact manifolds also have open book decompositions \cite{Giroux_ICM} with monodromy given by a symplectomorphism of a high-dimensional Weinstein domain. But for our proof of Corollary \ref{cor: cobordant_to_fillable}, we do not need know anything about the high-dimensional symplectic mapping class group  (which is not known in any case). Furthermore, the 
high-dimensional group is not generated by Dehn twists, which was a key ingredient in the proof in \cite{Etnyre_Honda_caps}. Instead, we can reverse the perspective 
and try to use Corollary \ref{cor: cobordant_to_fillable} to deduce properties of the symplectic mapping class group.

For any compactly supported symplectomorphism $\phi$ of a Weinstein domain $W^{2n}$, there is a contact manifold $(Y^{2n+1}_-,\xi_-)$ that has an open book decomposition
with page $W^{2n}$ and monodromy $\phi$.
By Corollary \ref{cor: cobordant_to_fillable}, there is a Weinstein cobordism $C^{2n+2}$ such that $\partial_-C = (Y_-, \xi_-)$ and $\partial_+ C = (Y_+, \xi_+)$ has a Weinstein filling. 
By work of Giroux and Pardon \cite{Giroux_Pardon_Lef}, any Weinstein domain has a Lefschetz fibration and hence the Weinstein-fillable contact structure $(Y_+, \xi_+)$ has an open book decomposition whose monodromy is a product of positive Dehn twists. 
A relative form of Giroux and Pardon's result would imply that 
Weinstein cobordant contact structures $(Y_-, \xi_-), (Y_+, \xi_+)$ have open book decompositions such that the pages of $(Y_-, \xi_-)$ are Weinstein subdomains of the pages of $(Y_+, \xi_+)$ and that the monodromies of the two open books are related by positive Dehn twists. 
However, contact structures have infinitely many compatible open book decompositions. 
For $n = 2$, Giroux \cite{Giroux_ICM} showed that all open book decompositions of a given contact manifold are related by \textit{stabilization}: adding a $1$-handle to the page of the open book and modifying the monodromy by a positive Dehn twist about a circle passing through the new 1-handle exactly once. There is a similar stabilization operation in high-dimensions using high-dimensional Dehn twists but it is unknown whether all open book decompositions of a given contact manifold are related by stabilizations in this case. 
The above discussion shows that a relative version of Giroux and Pardon's result \cite{Giroux_Pardon_Lef} and a high-dimensional analog of Giroux's stabilization result \cite{Giroux_ICM}
would give a positive answer to the following question. 
\begin{question}\label{question: symplectomophism_twist}
Suppose $\phi$ is a compactly supported exact symplectomorphism of a Weinstein domain $W^{2n}$. Does there exist a Weinstein domain $X^{2n}$ containing $W^{2n}$ as a subdomain such that $\phi$, as a symplectomorphism of $X^{2n}$, is compactly symplectically isotopic to a product of (positive and negative) Dehn twists of $X^{2n}$? 
\end{question}
Here we view $\phi$ as a symplectomorphism of $X^{2n}$ by extending it by the identity over $X^{2n}\backslash W^{2n}$. Of course, the Weinstein domain $X^{2n}$ may depend on the symplectomorphism $\phi$; in some sense, 
the symplectic data of $\phi$ is transferred to the data of $X^{2n}$.

As we noted above,  Dehn twists do not generate the high-dimensional symplectic mapping class group. There is sometimes even a topological obstruction to this, as explained to us by Casals. For example, $(\mathbb{RP}^{2n+1}, \xi_{std})$ has an open book decomposition with page $T^*\mathbb{RP}^n$
and monodromy the \textit{fibered} Dehn twist $\phi$
of $T^*\mathbb{RP}^n$.
The symplectomorphism $\phi$ is not smoothly isotopic to the identity;
otherwise  
$\mathbb{RP}^{2n+1}$
would have a smooth filling with handles of index $n$ and less, which is impossible
\cite{Eliashberg_Kim_Polterovich_RP2n+1}. At the same time, 
$T^*\mathbb{RP}^n, n \ge 1,$ has no Lagrangian spheres to Dehn twist about and so the fibered Dehn twist $\phi$ cannot be smoothly isotopic to a product of Dehn twists. More generally, if we stabilize the open book decomposition
$(T^*\mathbb{RP}^n,\phi)$ 
in the sense of Giroux by adding $n$-handles to $T^*\mathbb{RP}^n$ and modifying $\phi$ by Dehn twists through these handles to get a new open book decomposition $(X^{2n}, \phi')$ of $\mathbb{RP}^{2n+1}$, 
the resulting symplectomorphism $\phi'$ of the new Weinstein page $X^{2n}$ also cannot be smoothly isotopic to a product of Dehn twists, even if Lagrangian spheres exist in $X^{2n}$; again, this would give a smooth filling of 
$\mathbb{RP}^{2n+1}$
with handles of index $n+1$ and less.
In Question \ref{question: symplectomophism_twist}, we add handles to the page without necessarily adding Dehn twists to the monodromy. In this case, there is no topological obstruction as shown by \cite{Bowden_Crowley_Stipsicz_stein_fillable_II}.

\subsection{Codimension 2 contact embeddings}

Pancholi and Pandit \cite{pancholi_contact_embeddings} recently proved an h-principle for codimension 2 contact embeddings using open book decompositions. 
We will give an alternative proof of their result motivated by the stacking construction and Weinstein hypersurfaces \cite{eliashberg_revisited}. 
Both our proof and the proof of Pancholi and Pandit rely on the h-principle for overtwisted contact structures \cite{BEMtwisted}; in addition, we also need the h-principle for loose Legendrians \cite{Murphy11}.
The following is Theorem \ref{thm: codimension_2} from the Introduction. 
\begin{theorem}\label{thm: codimension_22}
If $(Y^{2n-1}, \xi_0), n \ge 3,$ has a contact embedding into $(Z^{2n+1}, \xi)$ with trivial normal bundle and $(Y^{2n-1}, \xi_1)$ is almost contactomorphic to $(Y^{2n-1}, \xi_0)$, then $(Y^{2n-1}, \xi_1)$ also has a contact embedding into $(Z^{2n+1}, \xi)$. 
\end{theorem}
\begin{proof}
Let $(Y^{2n-1}, \xi_{ot})$ be the overtwisted contact manifold in the same almost contact class as $(Y^{2n-1}, \xi_0)$ and $(Y^{2n-1}, \xi_1)$.
The result of Casals, Murphy, and Presas \cite{CMP} shows that there are smoothly trivial Weinstein cobordisms $C_i^{2n}, i=0,1,$ such that $\partial_- C_i^{2n} = (Y, \xi_{ot})$ and  $\partial_+ C_i^{2n} = (Y, \xi_i)$.
More precisely, there are  Legendrians $\Lambda_i \subset (Y, \xi_{ot})  \cup H^{n-1}, i = 0,1,$ 
such that $C_i^{2n} =  (Y, \xi_{ot})  \cup H^{n-1} \cup H^n_{\Lambda_i}$. 
We can view the handles $H^n_{\Lambda_i}$ as Weinstein cobordisms $W_i^{2n} \subset C_i^{2n}$ such that  
$\partial_- W_i = (Y, \xi_{ot}) \cup H^{n-1}$ and $\partial_+ W_i = (Y, \xi_i)$. 
The co-cores of the handles $H^n_{\Lambda_i}$ are Lagrangian disks $D_i^n \subset W_i^{2n}$ with Legendrian boundary $\partial_+ D_i^n = \Lambda_{loose, unknot} \subset \partial_+ W_i^{2n} = (Y^{2n-1}, \xi_i)$. Similarly, the Lagrangian cores $L^n_i$ of $H^n_{\Lambda_i}$ are Lagrangian disks  in $W^{2n}_i$ with $\partial_- L^n_i = \Lambda_i \subset \partial_-W_i^{2n} = (Y, \xi_{ot}) \cup H^{n-1}$.

By assumption,  $(Y^{2n-1}, \xi_{0})$ has a contact embedding into $(Z^{2n+1}, \xi)$ with trivial normal bundle.
Hence $(Y^{2n-1}, \xi_0) \times [0,1]$ is a Weinstein hypersurface of $(Z^{2n+1}, \xi)$.
The Weinstein cobordism $C^{2n}_0$ is smoothly trivial and therefore is almost symplectomorphic to  $(Y^{2n-1}, \xi_0) \times [0,1]$; this almost symplectomorphism is the identity near $\partial_+ C^{2n}_0 = (Y^{2n-1}, \xi_0) \times \{1\} $. By applying this almost symplectomorphism to $D^n_0 \subset C^{2n}_0$, we get a formal Lagrangian disk $D^n_{0, formal}$ in $(Y^{2n-1}, \xi_0) \times [0,1]$ formally isotopic to $D^n_0$; near its boundary, $D^n_{0, formal}$ agrees with the genuine Lagrangian $\Lambda_{loose, unknot} \times [1-\epsilon, 1] \subset (Y, \xi_0) \times [1-\epsilon, 1]$.
Since we have a symplectic embedding of  $(Y^{2n-1}, \xi_0) \times [0,1]$ into $(Z^{2n+1}, \xi)$, we can lift the formal Lagrangian $D^n_{0,formal}$ in  $(Y^{2n-1}, \xi_0) \times [0,1]$ to a formal Legendrian in $(Z^{2n+1}, \xi)$ that also agrees with $\Lambda_{loose, unknot} \times [1-\epsilon, 1]$ near its boundary. 
By the existence part of Murphy's h-principle \cite{Murphy11}, there is a loose Legendrian disk $D^n_{0,loose}$ in $(Z^{2n+1}, \xi)$ that coincides with $\Lambda_{loose, unknot} \times [1-\epsilon, 1] \subset (Y, \xi_0) \times [1-\epsilon, 1] \subset (Z^{2n+1}, \xi)$ near its boundary.  
We note that $D^n_{0, loose} \subset (Z^{2n+1}, \xi)$ cannot be entirely contained in  
$(Y, \xi_0) \times [0,1] \subset (Z^{2n+1}, \xi)$ since otherwise the loose Legendrian $\Lambda_{loose, unknot} \subset (Y, \xi_0)$ would have a Lagrangian disk filling in the symplectization $(Y, \xi_0) \times [0,1]$, which is impossible in general.

 A neighborhood of $D^n_{0, loose}$ in $(Z^{2n+1}, \xi)$ looks like a neighborhood of $D^n$ in the 1-jet space $J^1(D^n) = T^*D^n \times \mathbb{R}$. The Weinstein hypersurface $T^*D^n \subset J^1(D^n)$ gives a corresponding Weinstein hypersurface $T^*D^n_{0, loose}$ in $(Z^{2n+1}, \xi)$ containing $D^n_{0, loose}$. Note that $T^*(S^{n-1} \times [1-\epsilon, 1]) \subset T^*D^n_{0, loose}$  lies inside the 
hypersurface $(Y^{2n-1}, \xi_0) \times [1-\epsilon,1]$; in fact, it coincides with $J^1(\Lambda_{loose ,unknot}) \times [1-\epsilon, 1] \subset (Y^{2n-1}, \xi_0) \times [1-\epsilon, 1]$, where $J^1(\Lambda_{loose ,unknot})$ is a neighborhood of 
$\Lambda_{loose ,unknot}$ in  $(Y^{2n-1}, \xi_0)$. 
So we can glue the Weinstein hypersurfaces $T^*D^n_{0, loose}$ and $(Y^{2n-1}, \xi_0) \times [1-\epsilon,1]$ in $(Z^{2n+1}, \xi)$ to get another Weinstein hypersurface in $(Z^{2n+1}, \xi)$. This hypersurface as an abstract Weinstein cobordism is precisely $W_0^{2n}$ since $W_0^{2n}$ is defined by the fact that $\partial_+ W_0 = (Y, \xi_0) $ and that it has a single co-core $D^n_0$ with $\partial_+ D^n_0 = \Lambda_{loose, unknot}$. The contact embedding of $\partial_+ W_0^{2n}$ into $(Z^{2n+1}, \xi)$ coincides with the original contact embedding of $(Y^{2n-1}, \xi_0)$ into $(Z^{2n+1}, \xi)$. The contact embedding of $\partial_- W_0^{2n}$ gives us a contact embedding of $(Y^{2n-1}, \xi_{ot}) \cup H^{n-1}$
 into $(Z^{2n+1}, \xi)$.

 Similarly, the Lagrangian core $L^n_1$ of $W_1^{2n}$ can be used to produce 
a formal Legendrian disk in $(Z^{2n+1}, \xi)$ that  agrees with $\Lambda_1= \partial_- L^n_1 \subset \partial_- W_1^{2n} = (Y, \xi_{ot}) \cup H^{n-1} \subset (Z^{2n+1}, \xi)$ near its boundary.
 Again by Murphy's existence h-principle \cite{Murphy11}, there is a genuine Legendrian disk $L^{n}_{1,loose}$ in $(Z^{2n+1}, \xi)$ that also agrees with $\Lambda_1 \subset  (Y^{2n-1}, \xi_{ot}) \cup H^{n-1}$ near its boundary.   We can glue  $(Y^{2n-1}, \xi_{ot}) \cup H^{n-1}$ and  
 $T^*L^n_1$ to get a Weinstein hypersurface embedding of $W_1^{2n}$ into $(Z^{2n+1}, \xi)$; since the core $L^n_{1, loose}$ of $W_1^{2n}$ is  a loose Legendrian in $(Z^{2n+1}, \xi)$,  $W_1^{2n}$ is a \textit{loose} Weinstein hypersurface in the sense of Eliashberg 
 \cite{eliashberg_revisited}.  
    The contact embedding of $\partial_- W_1^{2n}$ agrees with the contact embedding of $(Y, \xi_{ot}) \cup H^{n-1}$ into $(Z^{2n+1}, \xi)$ produced in the previous paragraph. The contact embedding of $\partial_+ W_1^{2n}$ gives us the desired contact embedding of $(Y^{2n-1}, \xi_{1})$
 into $(Z^{2n+1}, \xi)$.
\end{proof}

\section{Stacking Lagrangians}\label{sec: Lag_fillings}
\subsection{Maximal Lagrangians}
Now we prove some results about Lagrangians. The following result is the relative analog of Theorem \ref{thm: finitely_many_subdomains2}. 
\begin{corollary}\label{cor: max_lagrangians}
Suppose that $\phi_1, \cdots, \phi_k: L^n \hookrightarrow W^{2n}, n \ge 3,$ are formally isotopic regular Lagrangians that are closed or have non-empty Legendrian boundary. Then there is a Weinstein domain $X^{2n}$ almost symplectomorphic to $W^{2n}$, a regular Lagrangian $j: L^n \hookrightarrow X^{2n}$, and formally isotopic Weinstein embeddings $\psi_1, \cdots, \psi_k: W^{2n} \hookrightarrow X^{2n}$ such that $\psi_1 \circ \phi_1 = \cdots = \psi_k \circ \phi_k= j: L^n \hookrightarrow X^{2n}$.
\end{corollary}
\begin{proof}
The proof essentially is an application of Corollary \ref{cor: abstractly_symp}. If $L$ is closed, we can directly apply 
Corollary \ref{cor: abstractly_symp} with $W_0 = T^*L$.
If $L$ has Legendrian boundary, we slightly modify the proof of Corollary \ref{cor: abstractly_symp} (which is stated only for subdomains) by requiring that $C_{flex}$ be a flexible cobordism \textit{in the complement of} $\partial L \subset \partial T^*L$. 
\end{proof}
The Lagrangian $j(L)$ is maximal in the sense that it extends the Lagrangians $\phi_i(L)$, namely $j(L)|_{\psi_i(W)} = \phi_i(L)$. 
In general the $X \backslash \psi_i(W)$ are not  trivial Weinstein cobordisms for all $i$ and $X$ is not Weinstein homotopic to $W$. For if $X \backslash \psi_i(W)$ were a trivial Weinstein cobordism for all $i$, then the Viterbo transfer map would an isomorphism on wrapped Floer homology \cite{Ganatra_Pardon_Shende_ii} and hence $WH(L, L; X) \cong WH(L|_{\psi_i(W)}, L|_{\psi_i(W)}; \psi_i(W)) \cong WH(\phi_i(L), \phi_i(L); W)$ for all $i$;
 the fact that the Viterbo transfer map is just restriction is because $L$ is regular in $W$ and hence intersects $\partial W$ in at most one component \cite{Abouzaid_Seidel}. 
However there are $\phi_i$ such that 
$WH(\phi_i(L), \phi_i(L); W)$ are different for different $i$ and hence the cobordisms cannot all be trivial in general. In particular, the Weinstein embeddings $\psi_i$ will generally not be symplectically isotopic. Finally we note that a version of Corollary \ref{cor: maximal_strong_cobordism} holds even if $\phi_i$ map $L$ to different (but almost symplectomorphic) Weinstein domains $W_i$.

\subsection{Legendrians with many Lagrangian fillings}\label{ssec:Leg_many_fillings}

We can also use the stacking construction to produce Legendrians with many Lagrangian fillings. The first such Legendrians were  produced in \cite{Sabloff}. These were Legendrians in $(S^{2n-1}, \xi_{std})$ with arbitrarily many (but finitely many) different Lagrangian fillings in $B^{2n}_{std}$ for $n \ge 2$; for $n = 2$, the Legendrians are necessarily disconnected since $tb(\Lambda) = - \chi(L)$ 
 and hence the formal class of $\Lambda$ determines the genus of the Lagrangian filling   \cite{Chantraine_Lagrangian_cob}.
 The examples for $n \ge 3$ are obtained by spinning the examples for $n =2$ and hence have quite special topology. In this section, we construct Legendrians that have many fillings with prescribed topology. Like \cite{Sabloff}, we produce arbitrarily many but finitely many fillings. Unlike for contact manifolds, which can have infinitely many symplectic fillings, it is not known whether there exist Legendrians that have infinitely many fillings with different topology. 

To construct Legendrians with many Lagrangian fillings, we first need to prove the Legendrian version of Theorem \ref{thm: contact_cobordisms2}. 
\begin{theorem}\label{thm: leg_cobordisms}
	Suppose that $\Lambda_1^{n-1}, \cdots, \Lambda_k^{n-1} \subset (Y^{2n-1},\xi), n \ge 3,$ are formally isotopic Legendrians. Then there exists a smoothly trivial Weinstein cobordism $X^{2n}$  with $\partial_- X^{2n} = (Y, \xi)$ and smoothly trivial regular Lagrangian cobordisms $L_i^n \subset X^{2n}$ such that $\partial_- L_i^n = \Lambda_i$ and $\partial_+ L_i^n$ are all coincide. 
	\end{theorem}
\begin{proof}
The proof is similar to the proof of Theorem \ref{thm: contact_cobordisms2}, which uses the existence of smoothly trivial Weinstein cobordisms from overtwisted contact structures to arbitrary contact structures \cite{CMP}. Here we will need the following Legendrian analog, proven in \cite{Lazarev_Lagrangian_regular}: for any Legendrian $\Lambda^{n-1} \subset (Y^{2n-1}, \xi), n \ge 3$, there is a smoothly trivial regular Lagrangian cobordism $L^n$ in the trivial Weinstein cobordism $(Y, \xi) \times [0,1]$ such that $\partial_+ L^n = \Lambda^{n-1}$ and  $\partial_-L^n= \Lambda_{loose}^{n-1}$. Here $\Lambda_{loose}^{n-1} \subset (Y^{2n-1}, \xi)$ is the (unique) loose Legendrian in the same formal class as $\Lambda^{n-1}$. Said another way, there is a Weinstein cobordism $C^{2n}$ that is Weinstein homotopic to the trivial one such that $\Lambda_{loose}^{n-1} \subset \partial_- C^{2n} = (Y^{2n-1}, \xi)$ viewed as a Legendrian of $\partial_+ C^{2n}$  is Legendrian isotopic to $\Lambda^{n-1}$ under the non-trivial identification of $\partial_+ C^{2n}$ with $(Y^{2n-1}, \xi)$.
So from this point of view, 
$(C^{2n}, L^n)$ is just $(T^*(\Lambda_{loose}\times[0,1]) \cup H^{n-1} \cup H^n_{\Lambda'}, \Lambda_{loose} \times [0,1])$ for some loose Legendrian $\Lambda'$ which is symplectically linked with $\Lambda_{loose}$.

We apply this result to the Legendrians $\Lambda_i \subset (Y, \xi)$. Since $\Lambda_i$ are formally Legendrian isotopic, $\Lambda_{i, loose}$ are Legendrian isotopic by the h-principle for loose Legendrians \cite{Murphy11}; we will fix one representative $\Lambda_{loose}$ of these Legendrians. Then by \cite{Lazarev_Lagrangian_regular} there exist Weinstein cobordisms $C_i^{2n}$ with the following properties: $\partial_-C_i^{2n} = (Y, \xi)$, $C_i^{2n}$ are homotopic to trivial Weinstein cobordisms, and $C_i^{2n}$ contain trivial Lagrangian cobordisms  $\Lambda_{loose} \times [0,1] \subset C_i^{2n}$ such that $\Lambda_{loose} \times \{1\}$ viewed as a Legendrian of $\partial_+ C_i^{2n} = (Y^{2n-1}, \xi)$ is Legendrian isotopic to $\Lambda_i$. Next we form $X^{2n}:= Stack_{(Y, \xi)}(C_1^{2n}, \cdots, C_k^{2n})$. Even though each $C_i^{2n}$ is Weinstein homotopic to  $(Y, \xi) \times [0,1]$,  the stacked cobordism $X^{2n}$ might not be Weinstein homotopic to $(Y, \xi) \times [0,1]$; see Corollary \ref{cor: domain_flex_everything}. However since $C_i^{2n}$ are smoothly trivial and $n \ge 3$,  $X^{2n}$ is still smoothly trivial.

There is a Lagrangian cobordism $\Lambda_{loose} \times [0,k] \subset X^{2n}$. Its positive boundary $\Lambda_{loose} \times \{k\} \subset \partial_+ X^{2n}$ is not loose since the attaching spheres of $X^{2n}$ are symplectically linked with $\Lambda_{loose}$. 
For each $i$, we will construct Lagrangian cobordisms $L_i^{n} \subset X^{2n}$ from $\Lambda_i \subset (Y, \xi)$ to $\Lambda_{loose} \times \{k\} \subset \partial X^{2n}$; so $\partial_- L_i = \Lambda_i$ and $\partial_+ L_i$ will all agree with the fixed Legendrian  $\Lambda_{loose} \times \{k\}$ as desired. 
To construct $L_i^n$,  we first Weinstein homotope $X^{2n}= Stack_{(Y, \xi)}(C_1^{2n}, \cdots, C_k^{2n})$ to $C_i^{2n} \circ D_i^{2n}$ for the Weinstein cobordism $D_i^{2n}  = Stack_{(Y, \xi)}(C_1^{2n}, \cdots, C_k^{2n}) \backslash C_i$. By construction, $\Lambda_i \subset \partial_+ C_i^{2n}$ is Legendrian isotopic to $\Lambda_{loose} \times \{1\} \subset \partial_+C_i^{2n}$ in $\partial_+C_i^{2n}$. 
The Lagrangian cobordism $L_i^n \subset X^{2n}$ is defined by gluing the
trivial cobordism $\Lambda_i \times [0,1] \subset C_i^{2n}$ with the cobordism induced by the Legendrian isotopy 
from $\Lambda_i$ to $\Lambda_{loose}$ in $\partial_+ C_i^{2n}$ with the trivial cobordism $\Lambda_{loose} \times [1,k] \subset 
D_i^{2n}$. Then $\partial_- L_i = \Lambda_i \subset (Y, \xi)$ and $\partial_+ L_i = \Lambda_{loose} \times \{k\} \subset \partial_+D^n_i$. Now we Weinstein homotope $C_i^{2n} \circ D^{2n}_i$ back to $X^{2n}= Stack_{(Y, \xi)}(C_1^{2n}, \cdots, C_k^{2n})$. Since this Weinstein homotopy consists just of changing the order of handle attachment, the Legendrian $\Lambda_{loose} \times \{k\} \subset \partial_+D^n_i$ corresponds to the Legendrian $\Lambda_{loose} \times \{k\} \subset \partial_+X^{2n}$. So we can view $L_i^n$ as a regular Lagrangian in $X^{2n}$ with $\partial_- L_i^n = \Lambda_i$ and $\partial_+ L_i^n = \Lambda_{loose} \times \{k\}$ as desired.
Finally, $L_i^n$ is smoothly trivial since it is a concatenation of three smoothly trivial cobordisms. 
\end{proof}
\begin{remark}
Although $\Lambda_i$ are all Legendrian isotopic to $\Lambda_{loose} \times \{1\}$ in $\partial_+ C^{2n}_i$, they may not be Legendrian isotopic in $\partial_+ X^{2n}$ once we attach the rest of the handles of $X^{2n}$. This is why the Lagrangian cobordism $L_i$ uses the Legendrian isotopy  from $\Lambda_i$ to $\Lambda_{loose} \times \{1\}$ in $\partial_+C^{2n}_i$. 
\end{remark}

We do not know if it is possible to control the Weinstein homotopy type of $C^{2n}$ and require $C^{2n}$ to be Weinstein homotopic to the trivial Weinstein cobordism $(Y, \xi) \times [0,1]$. In principle, starting with different collections of Legendrians $\Lambda_1, \cdots, \Lambda_k \subset (Y, \xi)$ might lead to different (smoothly trivial) Weinstein cobordisms $X^{2n}$. 

Theorem \ref{thm: leg_cobordisms} can be used to convert Lagrangians with \textit{formally isotopic} Legendrian boundaries into Lagrangians with genuinely isotopic Legendrian boundaries. The following result is the Legendrian analog of Corollary \ref{cor: contact_fillings}. 

\begin{corollary}\label{cor: Lag_fillings}
	If $W^{2n}, n \ge 3,$ is a Weinstein domain and $L_i^{n} \subset  W^{2n}$ are exact Lagrangians with formally isotopic Legendrian boundaries, then there is a Weinstein domain $X^{2n}$ almost symplectomorphic to $W^{2n}$ and exact Lagrangians $K_i^{n} \subset X^{2n}$  formally isotopic to $L_i^{2n}$ with Legendrian isotopic boundaries. Furthermore, $W^{2n}$ is a subdomain of $X^{2n}$ and $K_i^n$ extends $L_i^n$ in the sense that $K_i^n|_{W} = L_i^n$. 
\end{corollary}
\begin{proof}
Applying Theorem \ref{thm: leg_cobordisms} to the formally isotopic Legendrians $\partial L_i^n \subset \partial W^{2n}$, we get a smoothly trivial Weinstein cobordism $C^{2n}$ with $\partial_- C^{2n} = \partial W^{2n}$ and regular Lagrangian cobordisms $J_i^n \subset C^{2n}$ such that $\partial_- J_i^{n} = \partial L_i^n$ and $\partial_+ J_i^n$ all coincide. Then $K_i^n := L_i^n \circ J_i^n \subset W^{2n} \circ C^{2n}  =: X^{2n}$ satisfies the desired conditions. 
\end{proof}
The Lagrangians $L_i^n$ do not have to be regular and the domain $W^{2n}$ does not have to be Weinstein. However if $W^{2n}$ is Weinstein, then so is $X^{2n}$ and if $L_i^n \subset W^{2n}$ are regular, so are $K_i^{n} \subset X^{2n}$.
Unlike in Corollary \ref{cor: max_lagrangians}, here we only need a \textit{single} Weinstein embedding $\psi: W^{2n} \hookrightarrow X^{2n}$ to make the claim that $K_i^n$ extends $L_i^n$, i.e. $\psi^{-1}(K_i^n) = L_i^n$. This is possible since we do not require the $K_i^n$ to coincide (and indeed they may have different smooth topology); we  only require their Legendrian boundaries $\partial K_i^n$ to coincide. 	
As in Theorem \ref{thm: leg_cobordisms}, we do not know whether it is possible to make $X^{2n}, W^{2n}$ be Weinstein homotopic. 

Applying the h-principle for flexible Lagrangians \cite{EGL}  to Corollary \ref{cor: Lag_fillings}, we can construct Legendrians with many Lagrangians fillings that have prescribed smooth topology. 
\begin{corollary}\label{cor: Lag_fillings_flexible}
Suppose that $W^{2n}, n \ge 3,$ is an almost Weinstein domain and $L_1^n, \cdots, L_k^n \subset W^{2n}$ are formal Lagrangians with formally isotopic boundaries. Then there exists a Weinstein domain $X^{2n}$ almost symplectomorphic to $W^{2n}$ and regular Lagrangians $K_i^{n} \subset X^{2n}$ formally isotopic to $L_i^n$ with Legendrian isotopic boundaries. 
\end{corollary}

Using Corollary \ref{cor: Lag_fillings_flexible}, we can produce Legendrians with many Lagrangian fillings with different topology. Now we show how to produce Legendrians with many fillings that are diffeomorphic and even formally isotopic but are not isotopic through Lagrangians with Legendrian boundary.  The first such examples were constructed in \cite{Ekholm_Honda_Kalman_Lag_cob}; these are Legendrian links  in $(S^3, \xi_{std})$ with smoothly isotopic Lagrangian fillings in $B^4$ that have increasing genus. In the following result, we produce some high-dimensional examples of Lagrangian disk fillings in an exotic cotangent bundle; it is possible to modify our construction to produce fillings and domains with much more general topology. 
\begin{corollary}\label{cor: exotic disks_fillings}
For all $k$, there exists a Weinstein domain $W^{2n}, n\ge 3,$  almost symplectomorphic to $T^*M^n$ that has $k$ formally isotopic regular Lagrangian disks with the same Legendrian boundary that are not isotopic through Lagrangians with Legendrian boundary.  
\end{corollary}
\begin{proof}
By considering the graphs of the differential of certain functions, 
Abouzaid and Seidel \cite{Abouzaid_Seidel} constructed infinitely many Lagrangian disks $D^n_i \subset T^*M^n_{std}$ with connected Legendrian boundary in $\partial T^*M^n_{std}$ such that 
their wrapped Floer homology with the zero-section $WH(D^n_i, M^n; T^*M^n_{std})$ is different for different $i$. In particular, these disks are not isotopic through Lagrangians with Legendrian boundary and it is likely that they do not have the same Legendrian sphere boundaries in $\partial T^*M^n_{std}$; we will modify these examples to make their Legendrian boundaries  coincide. Furthermore, the Lagrangian formal class of these disks is determined by their intersection number with the zero-section and hence it is easy to construct infinitely many such disks that are all formally Lagrangian isotopic. 
 
By Corollary \ref{cor: Lag_fillings}, there exists a Weinstein domain $W^{2n}$ almost symplectomorphic to $T^*M^n_{std}$
that contains regular Lagrangian disks $K_i^n$ with the same Legendrian sphere boundary in $\partial W^{2n}$.
Furthermore, $K_i^n$ extends $D_i^n$ in the sense that there is a single Weinstein embedding $\psi: T^*M^n_{std} \hookrightarrow W^{2n}$ such that $\psi^{-1}(K_i^{n}) = D_i^n \subset T^*M^n_{std}$ for all $i$. We will use $\psi(M^n) \subset W^{2n}$ as a test Lagrangian. The intersection points of $\psi(M^n), K_i^n \subset W^{2n}$ are contained in $\psi(T^*M_{std})$ since $\psi(M) \subset \psi(T^*M_{std})$.
Applying the no-escape lemma \cite{Abouzaid_Seidel} to $W^{2n}\backslash \psi(T^*M^n_{std})$, all Floer trajectories in $W^{2n}$ asymptotic to intersection points of $\psi(M^n), K_i^n \subset W^{2n}$ are also contained in $\psi(T^*M^n_{std}) \subset W^{2n}$. The no-escape lemma applies since the intersection points are contained in $\psi(T^*M_{std})$ and the Legendrian boundary $\partial D_i^{n} \subset \partial T^*M^n_{std}$ is connected.
Therefore $WH(\psi(M^n), K_i^n; W^{2n})$ is isomorphic to $WH(\psi(M^n), K_i^n|_{\psi(T^*M^n)}; \psi(T^*M^n_{std}))$. Since $\psi^{-1}(K_i^{n}) = D_i^n \subset T^*M^n_{std}$, the latter is isomorphic to $WH(M^n, D_i^n; T^*M^n_{std})$. By construction $WH(M^n, D_i^n; T^*M^n_{std})$ are all different and hence so are $WH(\psi(M^n), K_i^n; W^{2n})$. Therefore, $K_i^n \subset W^{2n}$ are not isotopic through Lagrangians with Legendrian boundary.  
\end{proof}
Corollary \ref{cor: exotic disks_fillings} can be generalized to produce Lagrangians fillings in Weinstein domains with essentially arbitrary topology. We just need some collection of formally isotopic Lagrangians with connected Legendrian boundaries and some closed test Lagrangian whose wrapped Floer homology with these  Lagrangians is different.

\subsection{Weinstein domains with many closed Lagrangian submanifolds}\label{ssec: domains_many_lag}

In Corollaries \ref{cor: Lag_fillings_flexible} and \ref{cor: exotic disks_fillings}, we constructed Legendrians with many Lagrangian fillings that either have different topology or are formally isotopic (and hence diffeomorphic) but not Hamiltonian isotopic. Now we consider the closed case. In Corollary \ref{cor: finitely_lagrangians}, in the Introduction, we constructed Weinstein domains with many closed, exact Lagrangians with different topology. We can also construct Weinstein domains with many closed, exact Lagrangians that are formally isotopic but not Hamiltonian isotopic. The following is Corollary \ref{cor: exotic_lagrangians} from the Introduction. 
 
\begin{corollary}\label{cor: exotic_lagrangians2}
For all $k$, there is a Weinstein domain  $W^{2n}, n \ge 3,$ almost symplectomorphic to $T^*M^n_{std}$ that has $k$ regular Lagrangians that are formally Lagrangian isotopic to $M^n_{std} \subset T^*M_{std}$ (via the almost symplectomorphism to $W$)
but are not Hamiltonian isotopic in $W^{2n}$. 
\end{corollary}
\begin{proof}
As in the proof of Corollary \ref{cor: exotic disks_fillings}, we again use the formally isotopic regular Lagrangian disks $D_i^n \subset T^*M_{std}$ produced in \cite{Abouzaid_Seidel} that have different  wrapped Floer homology with the zero-section  $WH(D^n_i, M^n_{std}; T^*M^n_{std})$. By Corollary \ref{cor: max_lagrangians}, 
 there exists a Weinstein domain $W^{2n}$ almost symplectomorphic to $T^*M^n_{std}$ with a single Lagrangian disk $D^n \subset W^{2n}$ with Legendrian boundary and $k$ Weinstein  embeddings $\phi_i: T^*M_{std} \hookrightarrow W$ such that $\phi_i^{-1}(D^n) = D^n_i \subset T^*M_{std}$.
 By the no-escape lemma \cite{Abouzaid_Seidel} (which requires the fact that $\partial D_i^n$ is connected),  
 we again have 
$WH(\phi_i(M^n_{std}), D^n; W^{2n}) \cong WH(\phi_i(M^n_{std}), D^n|_{\phi_i(T^*M)}; \phi_i(T^*M^n_{std}))$. Since 
$\phi_i^{-1}(D^n) = D^n_i \subset T^*M_{std}$, this is isomorphic to 
$WH(M^n_{std}, D^n_i; T^*M^n_{std})$, which are different by assumption.
Hence $WH(\phi_i(M^n_{std}), D^n; W^{2n})$ are also all different and so 
$\phi_i(M^n_{std})$ are not  Hamiltonian isotopic for different $i$. 
Furthermore, these Lagrangians are formally isotopic to the zero-section $M^n_{std} \subset T^*M^n_{std}$ since $\phi_i(M^n) \subset W^{2n}$ is obtained from $M^n_{std} \subset T^*M_{std}$ by adding a smoothly trivial Weinstein cobordisms to $T^*M_{std}$. 
\end{proof}
We note that the proof of Corollary \ref{cor: exotic_lagrangians2} is in some sense opposite to the proof of Corollary \ref{cor: exotic disks_fillings}. In the former, we used a single Lagrangian disk as a test Lagrangian to distinguish closed Lagrangians. In the latter, we used a single closed Lagrangian sphere as a test Lagrangian to distinguish Lagrangian disks (with the same boundary). 
We also note that for any $k \ge 2$, the domain $W^{2n}$ as constructed above is always exotic and not Weinstein homotopic to $T^*M^n_{std}$. Otherwise there would be (at least) two closed Lagrangians $M^n_1, M^n_2 \subset T^*M_{std}$ and a Lagrangian disk with Legendrian boundary $D^n \subset T^*M_{std}$ such that $WH(M^n_1, D^n; T^*M_{std}) \ne WH(M^n_2, D^n; T^*M_{std})$. However by  \cite{Abouzaid, FukSS}, all closed exact Lagrangian submanifolds of $T^*M_{std}$ are equivalent in the wrapped Fukaya category of $T^*M_{std}$ and hence must have isomorphic wrapped Floer homology with any other test Lagrangian. 
Finally, we observe that it is easy to modify the topology of the ambient Weinstein domain to produce more general examples. We can add any Weinstein cobordism $C^{2n}$  to $W^{2n}$ and $\phi_i(M^n) \subset W^{2n} \subset W^{2n} \circ C^{2n}$ will still not be Hamiltonian isotopic since they still have different wrapped Floer homology with $D^n \subset  W^{2n} \circ C^{2n}$, where we extend $D^n \subset W^{2n}$ trivially to $D^n \subset  W^{2n} \circ C^{2n}$.

\bibliographystyle{abbrv}
\bibliography{sources}

\end{document}